\documentclass[10pt,reqnol]{amsart}
\usepackage{amsfonts}
\usepackage{url,hyperref,multirow}
\usepackage{color}
\usepackage[dvipsnames]{xcolor}
\usepackage{cases}
\usepackage{subfigure}
\usepackage[ruled,vlined]{algorithm2e}
\usepackage{algorithmic}
\usepackage{cancel}

\usepackage{epsfig,amsmath,bm,caption2,epstopdf}
\usepackage{amssymb,version,graphicx,fancybox,mathrsfs,pifont,booktabs,mathtools}
\usepackage[marginparwidth=2.5cm, marginparsep=3mm]{geometry}

\catcode`\@=11 \theoremstyle{plain}
\@addtoreset{equation}{section}   
\renewcommand{\theequation}{\arabic{section}.\arabic{equation}}
\@addtoreset{figure}{section}
\renewcommand\thefigure{\thesection.\@arabic\c@figure}
\@addtoreset{table}{section}
\renewcommand\thetable{\thesection.\@arabic\c@table}

\newtheorem{thm}{\bf Theorem}[section]
\newtheorem{cor}{\bf Corollary}[section]
\newtheorem{prop}{Proposition}[section]

\newtheorem{lmm}{\bf Lemma}[section]

\newenvironment{lemma}{\begin{lmm}}{\end{lmm}}
\theoremstyle{remark}
\newtheorem{remark}{\bf Remark}[section]
\theoremstyle{definition}

\newtheorem{exm}{\bf Example} 

\def \ri {{\rm i}}
\def \rd {{\rm d}}

\def \cb {\color{blue}}
\def \cred {\color{red}}

\newcommand{\comm}[1]{ \marginpar{%
\vskip-\baselineskip 
\raggedright\footnotesize
\itshape\hrule\smallskip{\color{red}#1}\par\smallskip\hrule}}

\begin{document}
\baselineskip 14pt
\bibliographystyle{plain}

\title[Logarithmic Nonlinearity] {Error  Analysis of  A first-order IMEX Scheme for the  Logarithmic Schr\"odinger Equation}
\author[
    L. Wang,\;  J. Yan\, \&\, X. Zhang
	]{Li-Lian Wang${}^{\dag}$,\; Jingye Yan${}^{\ddag}$\;  and\; Xiaolong Zhang${}^{\S}$
		}
	\thanks{${}^{\dag}$Division of Mathematical Sciences, School of Physical
		and Mathematical Sciences, Nanyang Technological University,
		637371, Singapore. The research of the authors is partially supported by Singapore  MOE AcRF Tier 1 Grant:  RG15/21. Email: lilian@ntu.edu.sg.\\
		\indent ${}^{\ddag}$Corresponding author. College of Mathematics and Physics, Wenzhou University, Wenzhou, China, and
School of Mathematical Sciences, Jiangsu University, Zhenjiang,  China.
Email: yanjingye0205@163.com.\\
		\indent${}^{\S}$MOE-LCSM, CHP-LCOCS, School of Mathematics and Statistics, Hunan Normal University, Changsha, 410081, China.  The research of this author is supported in part by the National Natural Science Foundation of China (No.12101229) and the Hunan Provincial Natural Science Foundation of China (No.2021JJ40331). Email: xlzhang@hunnu.edu.cn.\\
		\indent  The second author would like to thank the support of NTU and part of this  work was done when she worked as a Research Fellow in NTU.
		The third author would like to acknowledge the support of China Scholarship Council (CSC, No. 202106720024) as postdoctor for the visit of NTU}


\begin{abstract}  The logarithmic Schr\"odinger  equation (LogSE) has  a logarithmic nonlinearity $f(u)=u\ln |u|^2$ that is not differentiable at $u=0.$ Compared with its counterpart with  a regular nonlinear term, it  possesses  richer and  unusual dynamics, though the low regularity of the nonlinearity brings about significant challenges in both analysis and computation. Among very limited numerical studies, the semi-implicit regularized  method  via regularising $f(u)$ as  $ u^{\varepsilon}\ln ({\varepsilon}+ |u^{\varepsilon}|)^2$ to overcome the blowup of $\ln |u|^2$ at $u=0$ has been investigated  recently in literature. With the understanding of $f(0)=0,$ we analyze the non-regularized first-order Implicit-Explicit (IMEX) scheme for the LogSE. We introduce some   new tools for the error analysis that include the characterization of the H\"older continuity of the logarithmic term, and a nonlinear Gr\"{o}nwall's inequality.
		We provide ample numerical results to demonstrate the expected convergence.
	We position this work as the first one to study the direct linearized scheme for the LogSE  as far as we can tell.
\end{abstract}
 \maketitle

\vspace*{-10pt}
\section{Introduction}
In this paper, we are concerned with the numerical solution and   related error analysis for the logarithmic Schr\"odinger equation  of the form
\begin{equation}\label{lognls0B}
\begin{cases}
\ri  \partial_t u(x,t)+ \Delta u(x,t)=\lambda u (x,t)\ln(|u(x,t)|^2), \quad   x\in  \Omega, \;\; t>0,\\[2pt]
u(x,t)=0,\quad  x \in {\partial \Omega},\;\; t\ge 0; 
\quad u(x,0)=u_0(x),\quad x\in \bar \Omega,
\end{cases}
\end{equation}
where $\ri=\sqrt{-1},$  $\Omega\subset {\mathbb R}^d ( d=1,2,3 )$ is a bounded domain with a smooth boundary, and $u_0$ is a given function with a suitable regularity. The LogSE originally  arisen from the modeling of  quantum mechanics  \cite{Bialynicki1976nonlinear,Bialynicki1979Gaussons} has found diverse applications in  physics and engineering (see, e.g.,    \cite{Avdeenkov2011quantum,Buljan2003Incoherent,Hansson2009propagation,Hefter1985Application,krolikowski2000unified,De2003logarithmic,Zloshchastiev2010Logarithmic}).
It has attracted much  attention in research of the PDE theory and numerical  analysis. Indeed,   the presence of the logarithmic nonlinear term brings about significant challenges for both analysis and computation, but in return gives rise to some unique dynamics that the Schr\"odinger equation with e.g., cubic nonlinearity may not have.
One feature  of  the logarithmic nonlinearity is the tensorization property (cf.\! \cite{Bialynicki1976nonlinear}): {\em if
$u_0(x)=\Pi_{j=1}^d u_{j0}(x_j)$ on a separable domain $\Omega=\bigotimes_{j=1}^d\Omega_j,$ then  $u(x,t)=\Pi_{j=1}^d u_{j}(x_j,t)$ where $u_j(x_j,t)$ satisfies the LogSE in one spatial dimension on $\Omega_j$ with the initial data: $u_j(x_j,0)=u_{j0}(x_j)$ for $1\le j\le d.$} We refer to Carles and Gallagher \cite{Carles2018Universal} for an up-to-date review of the mathematical theory  on the focusing case (i.e., $\lambda<0$) and for some new results for the defocusing case (i.e., $\lambda>0$). Although the ``energy''  is conserved  (cf. \cite{Cazenave2003semilinear,Cazenave1980}):
\begin{equation}\label{lognls0BA}
E[u](t):=\|\nabla u(\cdot, t)\|_{L^{2}\left(\Omega\right)}^{2}+\lambda \int_{\Omega}|u(x,t)|^{2} \ln |u(x,t)|^{2} {\rm d}x= E[u_0],\quad  t\ge 0,
\end{equation}
as with the usual Schr\"odinger equation,  it does not have a definite sign since the logarithmic term $\ln|u|^2$ can change sign within $\Omega$ as time evolves. Moreover, according to \cite{Carles2018Universal}, no solution is dispersive   for $\lambda<0,$ while for
$\lambda>0$, solutions have a dispersive behavior (with a nonstandard rate of dispersion).
It is noteworthy that there has been a growing recent interest in the LogSE with a potential $V(x)$, where $(\Delta +V) u$ is in place of $\Delta u$ in \eqref{lognls0B} (see, e.g., \cite{Ardila2020Logarithmic,Carles2021Logarithmic,Carles2021nonuniqueness,Chauleur2021Dynamics,Scott2018solution,Zhang2020Bound}).

The numerical solution of the LogSE is less studied largely due to the non-differentiability of the logarithmic nonlinear term $f(u)=u\ln|u|$ at $u=0.$
Note that the partial derivatives $\partial_t f(u)$ and $\partial_{x_j} f(u)$ blow up, whenever $u( x,t)=0$ even for a smooth solution $u.$ In fact, $f(u)$ only  possesses $\alpha$-H\"older continuity with $\alpha\in (0,1)$ (see Lemma \ref{Holderlmm} and Theorem \ref{N2L2} below).
To avoid the blowup of $\ln|u|$ as $|u|\to 0,$  Bao et al  \cite{Bao2019Error} first proposed the regularization of the logarithmic nonlinear term, leading to the  regularized logarithmic Schr\"{o}dinger equation (RLogSE):
\begin{equation}\label{Rlognls0C}
\begin{cases}
\ri  \partial_t u^{\varepsilon}(x,t)+ \Delta u^{\varepsilon}(x,t)=\lambda u^{\varepsilon} (x,t)\ln(\varepsilon+|u^{\varepsilon}(x,t)|)^2, \quad   x\in  \Omega, \;\; t>0,\\[2pt]
u^{\varepsilon}(x,t)=0,\quad  x \in {\partial \Omega},\;\; t\ge 0; 
\quad u^{\varepsilon}(x,0)=u_0(x),\quad x\in \bar \Omega,
\end{cases}
\end{equation}
where the regularization parameter   $0<\varepsilon\ll 1.$ It was shown therein that
{\em if $u_0\in H^2(\Omega),$
\begin{equation}\label{uepsest}
\|u^{\varepsilon}-u\|_{L^{\infty}(0, T ; L^{2}(\Omega))} \leq C_{1} \varepsilon, \quad\|u^{\varepsilon}-u\|_{L^{\infty}(0, T ; H^{1}(\Omega))} \leq C_{2} \sqrt \varepsilon,
\end{equation}
where the positive constants $C_1=C_1(|\lambda|, T,|\Omega|)$ and $C_2=C_2(|\lambda|, T,|\Omega|,\|u_{0}\|_{H^{2}(\Omega)}).$} Then the semi-implicit Crank-Nicolson leap-frog in time
and central difference in  space were adopted to discretize the RLogSE.
In Bao et al \cite{Bao2019Regularized},
 they further introduced the first-order Lie-Trotter splitting and Fourier spectral method  for the    RLogSE
 \eqref{Rlognls0C}, where the conservation of mass is preserved and the constraints for discrete parameters in \cite{Bao2019Error} could be relaxed. This regularized splitting method has been recently extended in Carles and Su \cite{Carles2022numerical} to numerically solve the LogSE with a harmonic potential studied in their work \cite{Carles2021Logarithmic}. The notation of regularization was also imposed at the energy level in Bao et al  \cite{Bao2022Error} that resulted in a  regularization different from  \eqref{Rlognls0C} in their first work \cite{Bao2019Error}. We also point out that in \cite{Li2019Numerical}, the regularized scheme  in \cite{Bao2019Error} was applied to the LogSE in an unbounded domain truncated  with an artificial boundary condition. It is important to remark that the error analysis in
 \cite{Bao2019Error,Bao2019Regularized} indicated the severer restrictions in discretization parameters or loss of order due to the  logarithmic nonlinear term.

 Although  $f'(u)$ blows up at $u=0,$    $f(u)=u\ln|u|$ is well-defined at $u=0$ (note: $f(0)=\lim\limits_{u\to 0} f(u)=0$). With this understanding, we propose to directly discretize the original  LogSE  \eqref{lognls0B}   without regularizing the  nonlinear term.  In the course of finalising this work, we realized that Paraschis and  Zouraris  \cite{Paraschis2023FixedpointLogSE}   analysed  the (implicit) Crank-Nicolson scheme
  for  \eqref{lognls0B} (without regularization), so it required solving a nonlinear system at each time step. Note that the fixed-point iteration was employed therein as the  non-differentiable logarithmic nonlinear term ruled out the use of Newton-type iterative methods.
  The implicit scheme inherits the property of the continuous problem  \eqref{lognls0B},  so in the error analysis,
  the nonlinear term  can be treated easily as a Lipschitz type in view of
  \begin{equation}\label{LipSC}
\big|\Im [(f(u)-f(v))(u-v)^*]\big| \leq |u-v|^{2}, \quad \forall u, v \in \mathbb{C},
  \end{equation}
  see 
  \cite[Lemma 1.1.1]{Cazenave1980}.

  As far as we can tell, there is no work on the semi-implicit  scheme for the LogSE   without regularization.  The main purpose is to employ and analyze  the first-order IMEX scheme for \eqref{lognls0B}, that is,
  \begin{equation}\label{CNscheme}
\begin{cases}
 \ri  D_{\tau} u^n(x)+ { \Delta   u^{n+1}(x)}=2\lambda f(u^n(x)), \quad   x\in  \Omega, \;\; n\ge 0,\\[2pt]
u^{n+1}(x)=0,\quad x \in {\partial \Omega}; 
\quad u^0(x)=u_0(x),\quad x\in \bar \Omega,
\end{cases}
\end{equation}
where  $\tau$ is the time step size,  $t_n=n\tau$ with $0\le n\le [T/\tau]:=N_t,$ and  given a sequence $\{v^n\}$ defined on $t_n,$ we denote
\begin{equation}\label{deltafD00}
\begin{split}
&  D_{\tau}v^{n}=\frac{v^{n+1}-v^{n}}{\tau}.
\end{split}
\end{equation}
Here we focus on this scheme for the reason that it is a relatively  simpler setting to outstand the key to  dealing with the logarithmic nonlinear term.
In space, we adopt the finite-element method for \eqref{CNscheme}.

The rest of the paper is organized as follows. In section \ref{Sect1}, we present the essential tools for the error analysis include the characterization of the H\"older continuity of the logarithmic nonlinear term, and a nonlinear Gr\"{o}nwall's inequality. In sections \ref{sect2}, we conduct error estimates for the LogSE.  The final section is for numerical results and discussions.

\medskip
\section{ H\"older continuity and a nonlinear Gr\"{o}nwall's inequality}\label{Sect1}

In this section, we characterize the H\"older continuity of the nonlinear functional $f(u)=u\ln |u|,$ and then present a useful nonlinear Gr\"{o}nwall's inequality. These results  are indispensable to the forthcoming analysis.
\subsection{H\"older continuity of  $f(u)=u\ln |u|$}
\begin{lemma}\label{Holderlmm}
Let $f (z)=z \ln |z|$ for  $z\in \mathbb C.$   If $|u|, |v|\ge  0,$ then we have
\begin{equation}\label{fuvl}
|f(u)-f (v)|\leq  { (|\!\ln y| +1)}\,  |u-v|,\quad y:=\max\{|u|,|v|\}.
\end{equation}
 Given $\alpha\in (0,1),$ denote $\delta_\alpha:=e^{\frac \alpha {\alpha-1}}.$   If  $|u|, |v|\in [0, \epsilon]$ and $0\le \epsilon\le \delta_\alpha,$  then  we have
\begin{equation}\label{holderbnd}
|f(u)-f (v)|
  \le {\mathcal H_\alpha}({\epsilon})  |u-v|^{\alpha},\quad {\mathcal H_\alpha}({\epsilon}):=(2\epsilon)^{1-\alpha}\big(|\!\ln \epsilon|+1\big),
\end{equation}
so  $f (z)$ is {\rm(}locally{\rm)} $\alpha$-H\"older continuous.
\end{lemma}
\begin{proof}
 We largely follow Alfaro and Carles  \cite{Alfaro2017Superexponential} and provide the proof for the readers' reference.
 In view of the symmetry, we only need to consider $|u|\ge |v|.$
If $|u|>|v|>0,$ then we have
\begin{equation*}
\begin{split}
|f(u)-f (v)|& =  \big|(u-v)\ln |u|+v(\ln |u|-\ln |v|)\big|
\\ &\le \Big\{ |u-v|\big|\!\ln |u|\big| + |v|  \ln \Big(1+\frac{|u|-|v|}{|v|}\Big)\Big\}
\\&\le  \big\{ |u-v|\big|\!\ln |u|\big| + |u|-|v|\big\}
 \le  (|\!\ln |u|| +1) |u-v| ,
\end{split}
\end{equation*}
where we used the property:  $\ln (1+x)<x$ for $x>0,$ and the fact: $ |u|-|v|\le |u-v|.$
It is evident that if $|u|=|v|,$ then
\begin{equation*}
\begin{split}
|f(u)-f (v)|& =  |\!\ln |u|| \,|u-v|\le (|\!\ln |u|| +1) |u-v|.
\end{split}
\end{equation*}
In view of $f(0)=0$, the inequality  \eqref{fuvl} apparently holds for $|v|=0.$

We now turn to the proof of \eqref{holderbnd}.  If $ |u|\ge |v|\ge 0,$ we obtain from \eqref{fuvl} immediately that
for any $\alpha\in (0,1),$
\begin{equation}\label{revised1}
\begin{split}
|f(u)-f (v)| & \leq  \big\{(|\!\ln |u||+1)  |u-v|^{1-\alpha}\big\}  {|u-v|^{\alpha}}\\
&\leq \big\{ (|\!\ln |u||+1) (2 |u|)^{1-\alpha}\big\} {|u-v|^{\alpha}}\,  =  {\mathcal H_\alpha}(|u|) {|u-v|^{\alpha}}.
\end{split}
\end{equation}
Note that for $x\in (0,1),$
$$ {\mathcal H_\alpha}(x)=(2x)^{1-\alpha}(|\!\ln x|+1) =(2x)^{1-\alpha}(1-\ln x), $$
 has a unique stationary point $x=\delta_\alpha:=e^{-\frac\alpha  {1-\alpha}}.$ Moreover, it
is monotonically increasing for $x\in (0,  \delta_\alpha],$ but decreasing for $ \delta_\alpha<x< 1$.  Thus with the understanding  ${\mathcal H_\alpha}(0)=0,$
we have ${\mathcal H_\alpha}(|u|)\le {\mathcal H_\alpha}(\epsilon)$ for   $0\le |v|\le |u|\le \epsilon\le \delta_\alpha.$
Accordingly,  the bound   \eqref{holderbnd} follows from  \eqref{revised1} and the symmetry in $u$ and $v.$
\end{proof}

With the aid of Lemma \ref{Holderlmm}, we can obtain the following $L^2$-bound. Throughout the paper, we denote the norm of
$L^p(\Omega)$ for $0<p\le\infty$ by $\|\cdot\|_{L^p(\Omega)},$ but for $p=2,\infty,$ we simply use the notation  $\|\cdot\|$ and  $\|\cdot\|_\infty.$
\begin{thm}\label{N2L2}
Let $f (u)=u \ln |u|$ be a composite function defined on $\Omega.$ Assume that  $u,v\in L^\infty(\Omega)$ and  denote
$\Lambda_\infty= \max\{ \|u \|_{\infty}, \|v \|_{\infty} \}.$
 For any $\alpha \in (0,1)$ and $\epsilon \in (0, \delta_\alpha]$ with   $\delta_\alpha=e^{\frac \alpha {\alpha-1}},$ we have the following bounds.
\begin{itemize}
\item[(i)]If $\Lambda_\infty >\epsilon,$   then
\begin{equation}\label{ful2est}
\begin{aligned}
\|f(u)-f(v)\|^2
& \le  
 {\mathcal H}^2_\alpha(\epsilon) \left\|u-v\right\|_{L^{2\alpha}(\Omega)}^{2\alpha} +\Upsilon^2(\epsilon, \Lambda_\infty) \, \|u-v\|^{2},
\end{aligned}
\end{equation}
where $ \mathcal{H}_\alpha(\epsilon)$ is given in \eqref{holderbnd}  and
\begin{equation}\label{Hupsilon}
\begin{split}
\Upsilon(\epsilon, \Lambda_\infty)= \max_{\epsilon\le y\le  \Lambda_\infty } \!\!\big\{|\!\ln y|+1\big\}.
\end{split}
\end{equation}
\item[(ii)] If $\Lambda_\infty \le \epsilon$, then
\begin{equation}\label{fuvHolder}
\| f(u) - f(v) \| \le \mathcal H _\alpha( \epsilon ) \| u - v \|_{ L^{2\alpha}(\Omega) }^{\alpha}.
\end{equation}
\end{itemize}
\end{thm}
\begin{proof}   We first consider $\Lambda_\infty>\epsilon,$ and  decompose the domain into $\Omega=\cup_{i=1}^4\Omega_{i}$, where
\begin{equation}\label{decomOmega}
\begin{aligned}
&\Omega_1=\{x\in \Omega: 0 \le  |v|, |u| \le \epsilon \},  &&\Omega_2=\{x \in \Omega : 0 \le |v| \le \epsilon, |u| > \epsilon\}, \\
& \Omega_3=\{x \in \Omega : 0 \le |u| \le \epsilon, |v| > \epsilon \}, \;\; &&\Omega_4=\{x\in \Omega:    |u|, |v| > \epsilon\}.
\end{aligned}
\end{equation}
We obtain from  \eqref{holderbnd} that
\begin{equation}\label{case1}
\begin{split}
\int_{\Omega_1} | f(u) - f(v) |^2\, \rd x & \le \mathcal{H}_\alpha^2(\epsilon)   \int_{ \Omega_1}  |u-v|^{2\alpha}\, \rd x=
 \mathcal{H}_\alpha^2(\epsilon) \left\| u - v \right\|_{ L^{2\alpha}(\Omega_1) }^{2\alpha}.
\end{split}
\end{equation}
From \eqref{fuvl}, we have
\begin{equation*}
\begin{aligned}
\int_{\Omega_2}|f(u)-f(v)|^2\,\mathrm{d}x & \le  \int_{ \Omega_2} ( | \!\ln |u|| + 1)^2  |u-v|^2\, \rd x \\
& \le \sup_{x\in \Omega_2} \{ |\! \ln |u(x)| | + 1\}^2 \int_{\Omega_2}  |u - v|^2\, \mathrm{d} x \\
&\le  \max_{\epsilon \le y \le \|u\|_{L^\infty(\Omega_2)}}\{|\! \ln y | + 1\}^2  \int_{\Omega_2}  |u - v|^2\, \mathrm{d} x \\
&\le\big(\max_{\epsilon \le y \le \Lambda_\infty} \{| \! \ln y | +1 \}\big)^2  \|u-v\|^2_{L^2(\Omega_2)},
\end{aligned}
\end{equation*}
and similarly,
\begin{equation*}
\begin{aligned}
\int_{\Omega_3}|f(u)-f(v)|^2\,\mathrm{d}x
&\le\big(\max_{\epsilon \le y \le \Lambda_\infty} \{| \! \ln y | +1 \}\big)^2  \|u-v\|^2_{L^2(\Omega_3)}.
\end{aligned}
\end{equation*}
For clarity, we define $w= \max\{ |u|, |v|\}$  on $\Omega_4,$ and note that
\begin{equation*}\label{Omega2}
\epsilon \le w(x) \le  \max\{ \|u\|_{L^\infty (\Omega_4)}, \|u\|_{L^\infty (\Omega_4)}\}\le \Lambda_\infty.
\end{equation*}
Thus  we derive from \eqref{fuvl} that
\begin{equation*}
\begin{aligned}
\int_{\Omega_4}|f(u)-f(v)|^2\,\mathrm{d}x   & \le \int_{ \Omega_4} ( | \!\ln w | + 1)^2  |u-v|^2\, \rd x \\
&\le   \sup_{x\in \Omega_4} \{ |\! \ln w(x) | + 1\}^2  \int_{ \Omega_4}  |u-v|^2\, \rd x   \\
 &\le \big(\max_{\epsilon \le y \le \Lambda_\infty} \{| \! \ln y | +1 \}\big)^2\,  \|u-v\|^2_{L^2(\Omega_4)}.
 \end{aligned}
\end{equation*}
Summing up the above ``local" bounds, yields
\begin{equation*}
\begin{aligned}
\|f(u)-f(v)\|^2 = \sum_{i=1}^4 \int_{\Omega_i} |f(u) - f(v)|^2 \mathrm{d}x \le     {\mathcal H}^2_\alpha(\epsilon) \left \|u-v\right\|^{2\alpha}_{L^{2\alpha}(\Omega)} +
\Upsilon^2({\epsilon},\Lambda_\infty)  \, \|u-v\|^{2},
\end{aligned}
\end{equation*}
which leads to \eqref{ful2est}.

We now turn to the second case. It is clear that if  $\Lambda_\infty \le \epsilon$, then $\Omega=\Omega_1,$  i.e.,
 $\Omega_i=\emptyset $ for $i=2,3,4$ in \eqref{decomOmega}. Therefore \eqref{fuvHolder} is a direct consequence  of
\eqref{case1}.
\end{proof}

\begin{remark}\label{regularCase} {\em  It is seen from  \eqref{fuvHolder} that under the condition: $\max\{ \|u \|_{\infty}, \|v \|_{\infty} \}\le \epsilon,$  the logarithmic nonlinear term is in the ``H\"older" regime.  It can be  in ``Lipschitz" regime solely, if $|u(x)|, |v(x)|>\epsilon>0.$ Indeed, we have $\Omega=\Omega_4,$ and the above proof implies
$$\|f(u)-f(v)\| \le
\Upsilon({\epsilon},\Lambda_\infty)  \, \|u-v\|,$$
for any $\epsilon>0.$ \qed
    }
\end{remark}

\subsection{A nonlinear Gr\"{o}nwall's inequality} In the error analysis, we shall use the following Gr\"{o}nwall's inequality in accordance with the $\alpha$-H\"older regularity.
It is noteworthy  that Roshdy and  Mousa \cite{Roshdy2011Discrete} presented a discrete inequality  of Gr\"{o}nwall-Bellman type from  \eqref{disgronwalleqs1}, but it is different from \eqref{disgronwalleqs20} below.  We therefore feel compelled to sketch its proof for the readers' reference.
\begin{lemma} \label{disgronwalls}
Let $c_1, c_2, c_3$ be  positive constants and let  $\alpha\in (0,1]$. Suppose that  a sequence $\{y(n)\}$ satisfies
\begin{equation}\label{disgronwalleqs1}
y(n) \leq  c_1 + c_2\sum^{n-1}_{m=0}\limits y^{\alpha}(m) + c_3\sum^{n-1}_{m=0}\limits y(m),\quad n\geq 1. 
\end{equation}
Then we have
\begin{equation}\label{disgronwalleqs20}
\begin{split}
y(n) & \leq c_1 \bigg(1+(c_1^{\alpha-1}c_2+c_3)\frac{(1+\alpha c_1^{\alpha-1}c_2+c_3)^n-1}{\alpha c_1^{\alpha-1}c_2+c_3}\bigg),\quad n\ge 1.
 \end{split}
\end{equation}
\end{lemma}
\begin{proof}
Denote
$$x(n):=c_2\sum^{n-1}_{m=0}\limits y^{\alpha}(m) + c_3\sum^{n-1}_{m=0}\limits y(m),$$
and note  $x(0)=0$. Then \eqref{disgronwalleqs1} reads $y(n)\le c_1+x(n),$ so we can derive
\begin{equation*}
\begin{aligned}
x(n+1)-x(n)&=c_2y^{\alpha}(n)+c_3y(n) \leq  c_2(c_1 + x(n))^{\alpha}+c_3(c_1 + x(n))\\
&\leq c_2(c_1^{\alpha} + \alpha c_1^{\alpha-1}x(n))+c_1c_3+c_3x(n)\\
&=(c_1^{\alpha}c_2 +c_1c_3)+(\alpha c_1^{\alpha-1}c_2+c_3) x(n),
\end{aligned}
\end{equation*}
where we used the inequality: $(1+z)^\alpha \le 1+\alpha z$ for $z\ge 0$ and  $\alpha\in (0,1].$ This implies
\begin{equation*}
x(n+1)\leq (c_1^{\alpha}c_2 +c_1c_3)+ (1+\alpha c_1^{\alpha-1}c_2+c_3)x(n).
\end{equation*}
For notational simplicity, we further define $\beta: =1+ \alpha c_1^{\alpha-1}c_2 +c_3$ and $z(n):=x(n)/\beta^n$. Write
\begin{equation*}
x(n+1)=\beta^{n+1}z(n+1)\leq {  c_1^{\alpha} c_2 + c_1 c_3  }+ \beta x(n)=c_1^{\alpha}c_2 +c_1c_3 + \beta^{n+1} z(n),
\end{equation*}
which implies
\begin{equation*}
z(n+1)-z(n)\leq\frac{c_1^{\alpha}c_2 +c_1c_3}{\beta^{n+1}},\;\;\; {\rm so}\;\;\;   z(n)\leq \sum^{n-1}_{m=0}\frac{c_1^{\alpha}c_2 +c_1c_3}{\beta^{m+1}}.
\end{equation*}
Hence, we deduce from \eqref{disgronwalleqs1} and $x(n)=\beta^n z(n)$ that
\begin{equation*}
\begin{aligned}
y(n)&\le c_1+{ x(n)}= c_1+ {   \beta^n z(n) }\le  c_1 +  (c_1^{\alpha}c_2 +c_1c_3) \beta^n \sum^{n}_{m=1}\limits\frac{1}{\beta^{m}}\\
&=c_1 +  (c_1^{\alpha}c_2 +c_1c_3)\frac{\beta^n-1}{\beta-1}=c_1\Big{(}1+(c_1^{\alpha-1}c_2 +c_3)\frac{(1+\alpha c_1^{\alpha-1}c_2+c_3)^n-1}{\alpha c_1^{\alpha-1}c_2 +c_3}\Big{)}.
\end{aligned}
\end{equation*}
This completes the proof.
\end{proof}

\begin{remark}\label{caseA} {\em Note that  $\alpha c_1^{\alpha-1}c_2+c_3\ge \alpha (c_1^{\alpha-1}c_2+c_3),$ so \eqref{disgronwalleqs20} implies
\begin{equation}\label{disgronwalleqs33}
\begin{split}
y(n) &
 \le c_1 \big(1-\alpha^{-1}+\alpha^{-1}(1+\alpha c_1^{\alpha-1}c_2+c_3)^n\big).
 \end{split}
\end{equation}
 If $\alpha=1$ and $c_2=0,$ it reduces to the usual linear Gr\"{o}nwall's inequality.
}
\end{remark}


\medskip
\section{Convergence of the first-order IMEX-FEM scheme}\label{sect2}

In this section, we describe the full discretization scheme for the LogSE \eqref{lognls0B} and conduct the convergence analysis.

\subsection{The full discretisation scheme}\label{fullDis}
For simplicity, we assume that $\Omega$ is a bounded and convex  polygonal domain with a  quasi-uniform triangulation $\Sigma_h$  with $h=\max_{\pi_h\in\Sigma_h} \limits\{\operatorname{diam} \pi_{h}\}.$
 Let $V_h^0\subset H_0^1(\Omega)$ be the finite element approximation space,
 consisting of piecewise continuous polynomials of degree $r\, (r \geq 1)$ on each element of $\Sigma_h.$ As usual, we denote the corresponding FE interpolation operator by $\mathcal{I}_h.$

The  full  discretization scheme for  \eqref{lognls0B} is to find $u_h^{n+1}\in V_h^0$
 for  $ {0\le n\le N_t-1}$ such that 
\begin{equation}\label{lognlsfulldisweak}
\begin{split}
& \ri(D_{\tau}u^{n}_{h},v_h)-(\nabla u_h^{n+1} ,\nabla v_h)=2\lambda ( f( u^{n}_{h}),v_h),  \quad\forall\, v_h \in  V_h^0,
\end{split}
\end{equation}
where $u_h^0= \mathcal{I}_h u_0$ is the FE interpolation of $u_0,$  and $D_{\tau}$ are defined in \eqref{deltafD00}. It is evident that at each time step, we only need to solve a linear Schr\"odinger equation.  Here, we restrict our attention to this first-order scheme  as
it is a simper setting to illustrate the key idea of dealing with the nonlinear term.

{ Using \cite[Lemma 3.1]{akrivis_fully_1991} and following the argument in \cite{Paraschis2023FixedpointLogSE} and \cite[sec. 2]{Bao2019Error}, we can show the unique solvability of \eqref{lognlsfulldisweak}
(see Appendix \ref{AppendixA} for the sketch of the proof).
\begin{prop}\label{ExistUnique} For any fixed $\tau,h>0$ and $ 0\le n\le N_t-1,$ the discretised problem \eqref{lognlsfulldisweak} has  a unique solution $ u_h^{n+1}  \in V_h^0.$
\end{prop}
}
 \subsection{Useful lemmas}
 Throughout this paper, we denote by $C$ a generic positive constant independent of  $h,\tau$ and $u.$
Define the Ritz projection operator $R_{h}: H_0^1(\Omega)\to V_h^0$ as in \cite[Chapter 1]{Thomee2006Galerkin}:
\begin{equation}\label{projectionu}
(\nabla R_hu,\nabla \phi) = (\nabla u,\nabla \phi), \quad \forall\, \phi \in  V_h^0.
\end{equation}
The following  estimates and inequalities  can be found from various resources, see e.g.,
 \cite[Lemma 1.1 and (1.11)]{Thomee2006Galerkin} and  \cite[Theorem 4.4.20]{Susanne2008Book}.
\begin{lemma}\label{FEMtheory}
 For any $u\in H^1_0(\Omega)\cap H^s(\Omega)$, 
 we have
\begin{eqnarray}
& \|R_{h}u-u \| + h\|\nabla (R_{h}u-u)\|  \leq Ch^{s}\|u\|_{H^s(\Omega)},  \;  &  1 \le s \le r+1, \label{uexbound}\\
&\| \mathcal{I}_h u-u \| + h \|\nabla (\mathcal{I}_h u-u) \| \leq Ch^s\|u\|_{H^s(\Omega)},  \;  &  1 \le s \le r+1, \label{uexbound1}\\
&  \|R_{h}u-u \|_\infty+\|  \mathcal{I}_{h} u -u\|_{ \infty} \le C h^{s-\frac{d}{2} } \| u \|_{H^s(\Omega)}, \; &  \frac{d}{2} \le s \le r+1. \label{Linftybound}
\end{eqnarray}
Moreover, there holds the inverse inequality
\begin{equation}\label{inverseineq}
\|\phi\|_{\infty} \leq C h^{-\frac{d}{2}}\|\phi\|, \quad \forall\, \phi \in V_h^0.
\end{equation}
\end{lemma}

The result on the interpolation in  \eqref{Linftybound} is stated in \cite[Theorem 4.4.20]{Susanne2008Book}. The same estimate for the Ritz projection $ \|R_{h}u-u \|_\infty$ can be derived straightforwardly from the inverse inequality \eqref{inverseineq}  and \eqref{uexbound}.

%


 The Ritz projection is almost stable in the  $L^{\infty}$-norm.
\begin{lemma}[see Lemma 5 in \cite{Leykekhman2016}] \label{FEMtheory2}
For any $u \in L^{\infty}(\Omega) \cap H_{0}^{1}(\Omega)$, we have
\begin{equation}\label{Ritzstable}
\big\|R_{h} u\big\|_{\infty} \leq C  \ell_{h}  \|u\|_{\infty},
\end{equation}
and 
\begin{equation}\label{ellh}
	\ell_{h} =
	\begin{dcases}
		1 + |\!\ln h|, & \;{\rm if} \;\; r =1,\\
		1, & \;{ \rm if}\;\; r \ge 2.
	\end{dcases}
\end{equation}

 \end{lemma}

We have the following bound of the truncation error and sketch the derivation in Appendix  \ref{AppendixC} for  better readability of the paper.
\begin{lemma}\label{Tnorder}
Let
\begin{equation}\label{ucnA}
u\in C^2([0,T];L^2(\Omega))\cap C^1([0,T];H^2(\Omega)),
\end{equation}
and  $\mathcal{T}^{n}$ be the truncation error given by
\begin{equation}\label{trunT}
\mathcal{T}^{n}:=\ri({D}_{\tau}u^{n}- u^{n}_{t})+\Delta({u}^{n+1} - u^n),
\end{equation}
where we understand
$$
{ u^{0}_{t}(x)=\partial_t u(x,0)=\ri (\Delta u_0(x)-\lambda u_0(x)\ln|u_0(x)|^2).}
$$
Then  for all $0\le n\le N_t-1,$ we have
\begin{equation}\label{tnest}
\|\mathcal{T}^{n}\|^2\le \frac{2}{3} \tau^2\|u_{tt}\|_{L^\infty([t_n,t_{n+1}];L^2(\Omega))}^2+2\tau^2 \| u_{t}\|_{L^\infty( [t_n,t_{n+1}];H^2(\Omega))}^2.
\end{equation}
\end{lemma}

\subsection{Error estimates}
 The key step of the analysis  is to show the uniform boundedness of the numerical solution.
\begin{thm}\label{enzhao0} Assume that the solution of the LogSE
\eqref{lognls0B} has the regularity
\begin{equation}\label{ucnAB}
u\in C^2([0,T];L^2(\Omega))\cap C^1([0,T];H^1_0(\Omega)\cap H^{r+1}(\Omega)),\quad r\ge 1.
\end{equation}
If $0<\tau,h^{r+1}\le e^{-1}$ and $\tau\le Ch^{d/2},$ then
 the solution of the scheme \eqref{lognlsfulldisweak} satisfies
\begin{equation}\label{logbndA}
 \|u_h^m\|_{ \infty }  \le C_{u}, \quad   m=0,1,\cdots, N_t,
\end{equation}
where  $C_u$ is a generic positive constant independent of $\tau,h,$ but depends on the norms of $u$ in \eqref{ucnAB} and
the parameters $|\lambda|, |\Omega|,  T.$
\end{thm}
\begin{proof} For notational convenience, we denote
\begin{equation}\label{eheh0}
\begin{split}
& u^n-u_h^n=e^{n}_h+\hat e^n_h, \quad e_h^n:=R_hu^{n}-u_{h}^{n}, \quad \hat e_h^n:=u^n- R_hu^{n},\\ 
& \mathcal{N}_h^{n}:=2f(R_hu^{n})-2f(u^n_h),\quad   \widehat{\mathcal{N}}_h^{n}:=2f(u^n)-2f(R_hu^{n}).
\end{split}
\end{equation}
From the LogSE \eqref{lognls0B} and the scheme \eqref{lognlsfulldisweak}, we obtain
\begin{equation}\label{lognlsextdis}
\ri(D_{\tau}u^{n},v_h)-(\nabla {u}^{n+1},\nabla v_h)=(2\lambda f( u^{n})+\mathcal{T}^{n},v_h),\quad n = 0,\ldots, N_t-1.
\end{equation}
In view of \eqref{projectionu}, it  can be rewritten   as
\begin{equation}\label{lognlspronum}
\begin{split}
\ri (D_{\tau}R_hu^{n},v_h)& -(\nabla R_hu^{n+1},\nabla v_h)=(2\lambda f(u^{n})+\mathcal{T}^{n},v_h)+\ri(D_{\tau}(R_hu^{n}-u^{n}),v_h).
\end{split}
\end{equation}
Subtracting \eqref{lognlsfulldisweak} from \eqref{lognlspronum} leads to the error equation
\begin{equation}\label{erroreqnA}
\ri (D_{\tau}e^{n}_h,v_h)-(\nabla e^{n+1}_h,\nabla v_h)=\lambda ( \mathcal{N}_h^{n}, v_h)+\lambda ( \widehat{\mathcal{N}}_h^{n}, v_h)+(\mathcal{T}^{n},v_h)-\ri(D_{\tau}\hat e_h^n,v_h),
\end{equation}
for all $v_h\in V_h^0$ and $ n = 0,\ldots, N_t-1.$

We proceed with the proof by mathematical induction.
From Lemma \ref{FEMtheory},
we have
\begin{equation}\label{checke0}
\begin{split}
\|{e}^0_h\|^2&=\|R_{h}u_{0}-u^{0}_h\|^2=\| R_{h} u_{0}-{\mathcal I}_h u_0\|^2\le
  Ch^{2(r+1)}\|u_0\|_{H^{r+1}(\Omega)}^2,
\end{split}
\end{equation}
and  
\begin{equation}\label{u0hest}
\begin{split}
\|u^0_h\|_{\infty}&\le \|	R_{h}u_{0}\|_{\infty}+\|{e}^0_h\|_{\infty}\le C(\|u_{0}\|_{H^2(\Omega) }+ h^{-\frac{d}{2}}\|{e}^0_h\|)
 \\& \le C(1 +h^{r+1-{d}/{2}})\|u_0\|_{H^{r+1}(\Omega)}\le C_{u_0}.
\end{split}
\end{equation}
Thus \eqref{logbndA} holds  for $m=0.$

We assume that   \eqref{logbndA} holds for all $0\le n\le m,$ and next prove  it holds  for $n=m+1$ too.
Taking $v_h={e}^{n+1}_h$ in \eqref{erroreqnA} yields
\begin{equation}\label{classe}
\begin{split}
\frac{\ri}{\tau} \| e_{h}^{n+1} \|^2   - \frac{\ri }{\tau} (e_{h}^{n}, e_{h}^{n+1} )  -\|\nabla {e}^{n+1}_h\|^2  =\lambda(\mathcal{N}_h^{n},{e}^{n+1}_h)+\big(\lambda\widehat{\mathcal{N}}_h^{n}+\mathcal{T}^{n} -\ri D_\tau\hat e_h^n,{e}^{n+1}_h\big).
\end{split}
\end{equation}
Its imaginary part reads
\begin{equation}\label{realpartA}
\begin{split}
\frac{ \|e^{n+1}_h\|^2 -\|e^{n}_h\|^2 }{2\tau} + \frac{\| e_h ^{n+1}- e_{h}^{n} \|^2 }{2\tau} =\lambda \Im(\mathcal{N}_h^{n}, {e}^{n+1}_h)
+ \Im  \big(\lambda\widehat{\mathcal{N}}_h^{n}+\mathcal{T}^{n} -\ri D_\tau\hat e_h^n, {e}^{n+1}_h\big),
\end{split}
\end{equation}
where we used the identity
\begin{equation*}
	2\Re(e_h^{n}, e_{h}^{n+1} ) =     \| e_{h}^{n+1} \|^2 + \| e_{h}^{n} \|^2 -\| e_{h}^{n+1} - e_{h}^{n} \|^2.
\end{equation*}

We now deal with the first term on the right-hand side  of \eqref{realpartA}. Using the property \eqref{LipSC} and the imbedding inequality:
 $$ { \| u  \|_{L^{2\alpha}(\Omega) } \le |\Omega|^{\frac{1}{2\alpha}-\frac 1 2} \| u \|,} $$
{ for $\frac{1}{2} \le \alpha < 1,$} we derive  from the Cauchy-Schwarz inequality and Theorem \ref{N2L2} that for $0<\epsilon\le  \delta_\alpha$,
\begin{equation}\label{Nest}
\begin{split}
&| \lambda \Im(\mathcal{N}_h^{n}, {e}^{n+1}_h)| \le |\lambda| \cdot |\Im( \mathcal{N}_{h}^{n},e_h^{n}  ) | + |\lambda|\cdot  |\Im( \mathcal{N}_{h}^{n},e_h^{n+1} - e_h^{n}  ) |  \\&\le 2|\lambda| \cdot  \| e_h^n\|^2 +  \frac{|\lambda|^2 \tau} 2 \| \mathcal{N}_{h}^{n} \|^2 + \frac{\| e_{h}^{n+1} - e_{h}^{n} \|^2 }{2\tau}  \\
&\le   { {2} |\Omega|^{1-\alpha} }|\lambda|^2 {\mathcal H}^2_\alpha(\epsilon) \,\tau \, \|e_h^n\|^{2\alpha} + 2 |\lambda| \big( |\lambda|\Upsilon^2(\epsilon,\Lambda_\infty) \,\tau +1 \big) \, \|e_h^n\|^{2}  + \frac{\| e_{h}^{n+1} - e_{h}^{n} \|^2 }{2\tau},
\end{split}
\end{equation}
where $ {\mathcal H}_\alpha( \epsilon)$ and $ \Upsilon(\epsilon,\Lambda_\infty) $ are defined in \eqref{holderbnd} and \eqref{Hupsilon} with
\begin{equation}\label{GmmA}
  \Lambda_\infty=\max_{0\le n\le m}\big\{\|u^n_h\|_\infty, \|R_hu^n\|_\infty\big\}.
\end{equation}

We next estimate the second term on the right-hand side of \eqref{realpartA}. Similarly,   we deduce
from the Cauchy-Schwarz inequality and Theorem \ref{N2L2}  that for $ 0<\hat\epsilon\le \delta_{\alpha}$,
\begin{equation*}\label{case200}
\begin{split}
 | \lambda \Im(\widehat{\mathcal{N}}_h^{n},{e}^{n+1}_h)|& \le |\lambda| \|\widehat{\mathcal N}_h^{n}\|\| e^{n+1}_h\|\le \frac{1}{8}\|{e}^{n+1}_h\|^2+2|\lambda|^2 \|\widehat{\mathcal{N}}_h^{n}\|^2\\
    &\le {\frac{1}{8}\|{e}^{n+1}_h\|^2+ 8|\Omega|^{1-\alpha} |\lambda|^2 {\mathcal H}^2_\alpha(\hat \epsilon) \|\hat e_h^n\|^{2\alpha} + 8|\lambda|^2  \Upsilon^2(\hat\epsilon, \hat\Lambda_\infty) \, \|\hat e_h^n\|^{2}},
  \end{split}
\end{equation*}
where
\begin{equation}\label{Gmma3}
\hat \Lambda_\infty=\max_{0\le n\le N_t }\big\{ \|u^n\|_\infty, \|R_h u^n\|_\infty \big\}.
\end{equation}
Using the  FEM approximation result \eqref{uexbound}, we can further obtain  
\begin{equation}\label{case20}
\begin{split}
| \lambda \Im(\widehat{\mathcal{N}}_h^{n},{e}^{n+1}_h)| &\le \frac{1}{8}\|{e}^{n+1}_h\|^2
+{8 {C} |\lambda|^2 \big\{|\Omega|^{1-\alpha}
{\mathcal H}^2_\alpha( \hat\epsilon) }\, h^{2\alpha(r+1)} \|u\|^{2\alpha}_{L^{\infty}([t_n, t_{n+1}];H^{r+1}(\Omega))}\\
&\quad+
{\Upsilon^2(\hat \epsilon,\hat\Lambda_\infty)}\, h^{2(r+1)}
\|u\|^{2}_{L^{\infty}([t_n, t_{n+1}];H^{r+1}(\Omega))}\big\}\\
& \le \frac{1}{8}\|{e}^{n+1}_h\|^2+ {  C_u }   \big\{ |\Omega|^{1-\alpha} {\mathcal H}^2_\alpha(\hat \epsilon)h^{2\alpha(r+1)}+   \Upsilon^2(\hat\epsilon,\hat\Lambda_\infty)h^{2(r+1)} \big\},
\end{split}
\end{equation}
where the constant $C_u$ depends on the norms of $u$  but does not depend on $\tau,h,\alpha.$
On the other hand,  we derive from Lemma \ref{Tnorder} that
\begin{equation}\label{lastest2}
\begin{split}
&|\Im (\mathcal{T}^{n}, e^{n+1}_h)| \le \frac{1}{8}\|{e}^{n+1}_h\|^2 + 2 \| \mathcal{T}^{n} \|^{2}
\\ &\quad  \le \frac{1}{8}\|{e}^{n+1}_h\|^2 + { 4 } \tau^2\Big( \frac{1}{3} \| u_{tt} \|^2_{L^{\infty} ([t_n, t_{n+1}]; L^{2}(\Omega) )} + \| u_{t} \|^2_{ L^{\infty}([t_n, t_{n+1}]; H^{2} (\Omega) ) }\Big) \\
&{\quad = \frac{1}{8}\|{e}^{n+1}_h\|^2 + C_u \tau^2.}
\end{split}
\end{equation}
Similarly, we have
\begin{equation}\label{lastest}
\begin{split}
|\Re(D_\tau\hat e_h^n, {e}^{n+1}_h)|& \le  \frac{1}{8}\|{e}^{n+1}_h\|^2+ 2 \|D_\tau\hat e_h^n\|^2
\\& \le \frac{1}{8}\|{e}^{n+1}_h\|^2
  + { 2C} h^{2(r+1)}\|u_t\|^2_{L^\infty([t_n, t_{n+1}];H^{r+1}(\Omega))}\\
  &{ \le  \frac{1}{8}\|{e}^{n+1}_h\|^2+ C_u h^{2(r+1)}.}
\end{split}
\end{equation}

Combining \eqref{realpartA} with \eqref{Nest}-\eqref{lastest}, and eliminating the common term $ \| e_{h}^{n+1} - e_{h}^{n} \|^2/ (2\tau)$ on both sides, leads to
\begin{equation*}\label{lastest01}
	\begin{split}
		 \|e^{n+1}_h\|^2 - \|e^{n}_h \|^2
		& \le  \frac{3}{4} \tau   \| e_{h}^{n+1} \|^2 + 4|\Omega|^{1-\alpha} |\lambda|^2
		\mathcal{H}^2_\alpha(\epsilon) \tau^2  \| e_h^n \|^{2\alpha}
		\\&\quad + 4|\lambda| \tau  \big(|\lambda| \Upsilon^2(\epsilon,\Lambda_\infty )\tau  +1 \big) \| e_h^{n} \|^2
		\\&\quad  + { C_u}\,\tau \big(\tau^2+   h^{2(r+1)}+h^{2\alpha(r+1)}{ \mathcal H }^2_\alpha(\hat \epsilon ) +  h^{2(r+1)}  \Upsilon^2(\hat \epsilon,\hat\Lambda_\infty ) \big) .
	\end{split}
\end{equation*}
Summing up both sides of the above inequality  from $n=0$ to $m,$ gives
\begin{equation}\label{realpart2}
	\begin{split}
		\|e^{m+1}_h\|^2  &\le  \|e^{0}_h\|^2 +  \frac{3}{4} \tau \| e_{h}^{m+1} \|^2 +  4|\Omega|^{1-\alpha} |\lambda|^2
		\mathcal{H}^2_\alpha(\epsilon)  \tau^2  \sum_{n=0}^{m} \| e_h^n \|^{2\alpha}
		\\ &\quad + \tau  \Big( \frac{3}{4} + 4 |\lambda| + 4 |\lambda|^2 \Upsilon^{2}(\epsilon,\Lambda_\infty) \tau    \Big) \sum_{n=0}^{m}  \| e_{h}^{n} \|^2
		\\&\quad  + { C_u}\, T\big(\tau^2+ {{h^{2(r+1)}+h^{2\alpha(r+1)}{ \mathcal H }^2_\alpha(\hat \epsilon) +  h^{2(r+1)}  \Upsilon^2(\hat \epsilon,\hat\Lambda_\infty) \big)}}.
	\end{split}
\end{equation}
 Then  we obtain from \eqref{checke0} and \eqref{realpart2} that for   $0 <\tau < 4/3,$
\begin{equation}\label{Upsilonem100}
\|e^{m+1}_h\|^2\le  c_1 + c_2 \sum_{n=0}^{m}\|e^n_h\|^{2\alpha}
+  c_3  \sum_{n=0}^{m}  \|e^{n}_h\|^2,
\end{equation}
{ where $\alpha\in [1/2,1)$} and
\begin{equation}\label{c123}
\begin{split}
c_1&=  { C_u \, C_\tau}\, T\big(\tau^2+ {{h^{2(r+1)}+h^{2\alpha(r+1)}{ \mathcal H }^2_\alpha(\hat \epsilon) +  h^{2(r+1)} \Upsilon^2(\hat \epsilon,\hat\Lambda_\infty) \big)}}, \\ 
c_2&= 4 C_\tau\, |\Omega|^{1-\alpha} |\lambda|^2   \tau^2 {\mathcal H}^2_\alpha(\epsilon),\quad c_3=\tau C_\tau\, \Big( \frac 3 4+ 4 |\lambda| +4 |\lambda|^2 \Upsilon^{2}(\epsilon,\Lambda_\infty) \tau \Big),
\end{split}
\end{equation}
with $C_\tau=1/(1-3\tau/4).$
Using the nonlinear Gr\"{o}nwall's inequality in Lemma \ref{disgronwalls} leads to
\begin{equation}\label{disgronwalleqs200}
\begin{split}
\|e^{m+1}_h\|^2
& \leq c_1 \big(1-\alpha^{-1}+\alpha^{-1}{(1+\alpha c_1^{\alpha-1}c_2+c_3)^m}\big)\\
&\le c_1\big(1-\alpha^{-1} +\alpha^{-1} {\rm exp}\big({m (\alpha c_1^{\alpha-1}c_2+c_3)}\big) \big),
\end{split}
\end{equation}
where we used the inequality: $(1+z)^m\leq e^{mz}$ for $z\ge 0.$
Note that for $0<\epsilon,\hat \epsilon\le \delta_\alpha,$  the factor in the exponential can be bounded by
\begin{equation}\label{Bronwalleqs00200}
\begin{split}
&G_m:={m \big(\alpha { c_2 }/{ c_1^{ 1 - \alpha } }+c_3\big)}= m\tau  C_\tau \big( 3/4+ 4 |\lambda| +4 |\lambda|^2 \Upsilon^{2}(\epsilon,\Lambda_\infty) \tau \big)\\[6pt] &\quad +  \frac{ 4 m \alpha C_\tau |\Omega|^{1-\alpha} |\lambda|^2   \tau^{2\alpha} {\mathcal H}^2_\alpha(\epsilon)}{({ C_u C_\tau}\,T)^{1-\alpha} \big(1+ \tau^{-2} h^{2(r+1)}+\tau^{-2}h^{2\alpha(r+1)}{ \mathcal H}^2_\alpha
(\hat\epsilon)+ \tau^{-2}h^{2(r+1)}  \Upsilon^2(\hat \epsilon, \hat\Lambda_\infty) \big)^{1-\alpha} }
\\[6pt]
&\le  C_\tau T \big( 3/4+ 4 |\lambda| +4 |\lambda|^2 \Upsilon^{2}(\epsilon,\Lambda_\infty) \tau \big) +
4 \alpha { (C_u C_\tau)^{\alpha-1}} T^\alpha  |\Omega|^{1-\alpha}  |\lambda|^2   \tau^{2\alpha-1} {\mathcal H}^2_\alpha(\epsilon),
\end{split}
\end{equation}
where  we used the inequality: $(1+z)^{\beta}\le 1$ for $z\ge 0$ and $\beta=\alpha-1<0.$

We now choose suitable  values of $\epsilon$ and $\hat \epsilon$ in the bound of $G_m$ and $c_1,$  respectively,  so as to obtain the best possible order of convergence.
Taking $\epsilon=\tau$ and using the induction assumption, we find from the definition \eqref{Hupsilon} and \eqref{GmmA}  that 
 for $0<\tau  \le \delta_\alpha,$
 \begin{equation*}\label{UptauA00}
 \begin{split}
& \tau^{2\alpha-1} {\mathcal H}^2_\alpha(\tau)\le 4^{1-\alpha} \tau (|\!\ln \tau|+1)^2,
 \\[3pt]
&  \tau \Upsilon^2(\tau, \Lambda_\infty)\le 2\tau \big( | \! \ln  \tau|^2+ (|\!\ln  \Lambda_\infty | +1)^2 \big) \le 2\tau \big( | \! \ln  \tau|^2+ C_u \big)
\end{split}
 \end{equation*}
 where we used $\| R_h u^n \|_\infty \le C h^{1/2}\| u^n \|_{H^2 ( \Omega ) } $ from \eqref{Linftybound}.
 Note that for $0<\tau  \le \delta_\alpha<1,$
   $$1\le C_\tau\le C_{\delta_\alpha}\le \frac 1{1-3/(4e)}< \frac {7} 5.$$
Thus we derive from \eqref{Bronwalleqs00200}  and the above  that
 \begin{equation}\label{GmC}
 \begin{split}
  G_m  \le \frac {7} 5 T\Big(\frac 3 4+ 4 |\lambda| \Big) + CT \big(1+(T C_u)^{\alpha-1}\big) \tau |\ln\tau|^2
  +C_u\tau.
  \end{split}
 \end{equation}
 Recall that $C_u$ is in terms of the norms of $u$ with regularity given in \eqref{ucnAB}, which is nonzero.

 We now deal with $c_1$ in \eqref{c123}. Taking ${\hat \epsilon=h^{r+1} }$, we find from \eqref{Hupsilon} and  \eqref{Gmma3}
 that
 \begin{equation}\label{est:c1}
 \begin{split}
  c_1 & \le  C_u C_\tau T\big\{ \tau^2 + h^{2(r+1)} +  4^{1-\alpha} h^{2(r+1)} ((r+1)|\! \ln h | +1 )^2  \\
& \quad +  2h^{2(r+1)} \big((r+1)^2|\!\ln h|^2 +  (|\!\ln \hat \Lambda_\infty | +1)^2 \big)\big\} \\
& \le C_u T \big\{ \tau^2 + 2(r+1)^2(4^{1-\alpha} +1) h^{2(r+1)} |\!\ln h|^2  \\
&\quad + \big(1 + 2^{3-2\alpha} + 2(|\!\ln \hat \Lambda_\infty|  + 1 )^2 \big)h^{2(r+1)} \big\} \\
& \le C_u T\big\{ \tau^2 +  h^{2(r+1)} |\! \ln h|^2 \},
\end{split}
 \end{equation}
where we used Lemma \ref{FEMtheory2} to bound
$|\!\ln \hat \Lambda_\infty|. $

We deduce from \eqref{c123}-\eqref{est:c1} that if {  $\frac 1 2\le \alpha<1,\, 0<\tau \le  \delta_\alpha$ and $h^{r+1} \le  \delta_\alpha,$ then}
 \begin{equation}\label{Gmbnd2}
 \begin{split}
 \|e^{m+1}_h\|&\le \sqrt{c_1(1 +\alpha^{-1}({\rm exp}(G_m)-1))}\\
 &\le  C_u  \, {\rm exp}\big(CT\{1+C_{u}^{\alpha-1}\tau | \! \ln \tau|^2\}\big) (\tau + { h^{r+1} } |\!\ln h|),
\end{split}
\end{equation}
where  the constants $C, C_u$ are essentially independent of $\alpha.$  Note that $\delta_\alpha={\rm exp}(\alpha/(\alpha-1))$ is decreasing in $\alpha$ and $\delta_\alpha\to 0$ as $\alpha\to 1^{-}.$ Thus we choose $\alpha=1/2$ so that $\delta_\alpha$ has the maximum value $e^{-1}.$ Moreover, as the function $\tau|\! \ln \tau |^2 $  is monotonically  increasing on $[0, \delta_\alpha] $, we have
$$
0\le \tau|\! \ln \tau |^2\le \delta_\alpha |\! \ln \delta_\alpha |^2\le \delta_{1/2} |\! \ln \delta_{1/2}|^2=e^{-1}.
 $$
In view of these, we infer from \eqref{Gmbnd2} that   if $0<\tau, h^{r+1} \le  e^{-1},$ then
 \begin{equation}\label{Gmbnd20}
\|e^{m+1}_h\| \le  C_u  (\tau + {h^{r+1} } |\!\ln h|).
\end{equation}
%
%
%
%
Finally, we obtain from \eqref{Linftybound}, \eqref{inverseineq} and \eqref{Gmbnd20} that if $\tau \le Ch^{d/2},$ then
\begin{equation}\label{umbnde00}
\begin{split}
\|u^{m+1}_h\|_\infty  & \le  \| R_h u^{m+1} \|_{ {\infty} } + \|e^{m+1}_h\|_{\infty}
 \le   C\big(\| u^{m+1} \|_{ H^2(\Omega) } +  h^{-d/2}\|e^{m+1}_h\|\big) \\
& \le  C_u  (1+\tau h^{-d/2} +  h^{r+1-d/2}|\! \ln h| ) \le C_u.
\end{split}
\end{equation}
The completes the mathematical induction.
\end{proof}

\begin{remark}\label{remarkadded} {\em It is noteworthy that the proof and results of Theorem   {\rm \ref{enzhao0}}  cover the special cases:  {\rm (i)}  $\Upsilon(\cdot,\cdot)=0$ and {\rm (ii)} $ {\mathcal H}_\alpha(\cdot)=0,$ which correspond to    Case-{\rm(ii)} in Theorem {\rm \ref{N2L2}}  and Remark {\rm \ref{regularCase}}, respectively.
\qed}
\end{remark}

We are now ready to present  the main result on the convergence of the numerical scheme.
\begin{thm}\label{mainresultA} Assume the conditions in Theorem  {\rm \ref{enzhao0}} hold.   Then we have the error estimate
\begin{equation}\label{L2est}
\|u^n-u_h^n\|\le C_u (\tau + { h^{r+1} } |\!\ln h|),\quad n=0,\cdots, N_t,
\end{equation}
 where $C_u$ is a positive  constant of the same type as in  Theorem  {\rm \ref{enzhao0}}.
\end{thm}
\begin{proof}  In view of \eqref{eheh0},
using  the estimate \eqref{Gmbnd20} in the proof of Theorem \ref{enzhao0} and the approximation results in Lemma \ref{FEMtheory}, we obtain immediately that
\begin{equation*}\label{step11}
\begin{split}
\|u^n-u_h^n\|\le \|e_h^n\|+\| R_h u^n-u^n\|\le   C_u(\tau + { h^{r+1}  }|\!\ln h| +h^{r+1}),
\end{split}
\end{equation*}
which completes the proof.
%
%
\end{proof}

Some remarks on Theorems \ref{enzhao0}-\ref{mainresultA} are  in order.
\begin{itemize}
\item[(i)] Compared with the usual Schr\"{o}dinger equation with a smooth nonlinear term, the non-differentiable logarithmic nonlinearity   brings about the extra logarithmic factor in \eqref{L2est} and a minor reduction of the order (i.e., ${ h^{ r+1} |\!\ln h| }$ versus $h^{r+1}$).  This seems unavoidable though such a factor is not  significant.
\smallskip
\item[(ii)]  The  condition  $\tau=O(h^{d/2})$ is resulted from the use of inverse inequality  \eqref{inverseineq} in Lemma \ref{FEMtheory}. For the Schr\"odinger equation with cubic nonlinearity,  such a constraint can be removed by using the argument of Li and Sun \cite{Li2013Error,Li2013Unconditional}. We refer to   Wang \cite{Wang2014A} where the nonlinear term $f(u)$ was assumed to be in $C^2(\mathbb R)$.  However, the argument therein  cannot be extended to this setting due to the non-differentiability of $f(u).$
\smallskip
\item[(iii)] The logarithmic nonlinearity also induces  a minor reduction of  convergence order 
with  the regularity assumption \eqref{ucnAB}.  Although  the theory on the regularity of  the solution to the LogSE is still under-explored,  the important family of soliton-like  solutions called Gaussons  (see  \cite{Bialynicki1979Gaussons,Carles2018Universal}) satisfy  \eqref{ucnAB}.
\smallskip
 \item[(iv)] In the error bound of the regularised method in Bao et al  \cite[Theorem 3.1]{Bao2019Error},
the regularity assumption was imposed on the solution of the regularised LogSE \eqref{Rlognls0C}, which is in contrast to  our results in the above.
\end{itemize}

\section{Numerical results}
In this section, we  provide ample numerical results to demonstrate the  convergence of the scheme analysed in   the previous section.

We first consider the LogSE \eqref{lognls0B} with the exact Gaussian solution (see, e.g.,  \cite[(1.7)]{Bao2019Error}):
\begin{equation}\label{2dGex}
u(x,t)=b\exp\big(\ri\big(x\cdot\zeta  -(a+|\zeta|^2)t\big)+(\lambda/2)|x-2\zeta  t|^2 \big), \quad x \in \mathbb{R}^d, \;\; t \geq 0,
\end{equation}
where $a=-\lambda (d-\ln|b|^2)$, and  $ 0\not=b \in \mathbb{R},$ $\zeta\in \mathbb{R}^d$ are free to choose.
Accordingly, we take the initial value and boundary conditions {on $\partial \Omega$} from the exact solution.
 In the  test, we set $\Omega=(-1,1)^d, \lambda=-1$ and choose $\zeta=0$, $ b=1$ in \eqref{2dGex} for $0\le t\le T=1$.
  {To deal with the inhomogeneous Dirichlet boundary data, we simply subtract a polynomial from the solution  as usual (see, e.g.,  \cite[(3.25)]{LiWangLi23}).}  To demonstrate the convergence rate in time, we fix $h=2^{-5}$ and take the time step as $\tau_j=0.1\times2^{-j},\,j=1,\ldots,5$ for the linear and quadratic FEM in 1D  (see Figure \ref{logsetimeorderFEM1D}). To examine the convergence order  in space, we  fix $\tau=10^{-5}$ and vary the
spatial mesh size $h_j=2^{-j},\,j=1,\ldots,5$ for the linear FEM and $h_j=\frac{1}{j+1},\,j=1,\ldots,5$ for the quadratic FEM
(see Figure \ref{logsespaceorderFEM1D}).

In {Figures \ref{logsetimeorderFEM1D} and \ref{logsespaceorderFEM1D}},  we plot  in log-log scale  the discrete $l^2$-error (denoted by $e_2$) and $l^\infty$-error (denoted by $e_\infty$)  of the scheme \eqref{lognlsfulldisweak} in 1D for the piecewise linear and quadratic FEM, respectively. For such a regular solution, we infer from Theorem \ref{mainresultA} (with $\alpha$ being sufficiently close to $1$) that  the expected convergence is approximately ${ O(\tau+h^\beta)}$ with $\beta=2,3$ in $e_2.$ As expected, the error $e_\infty$ is as good as the approximation of \eqref{2dGex} by FEM {in space} and FD in time with the above choice of $\tau$ { and $h$}.
\begin{figure}[!h]
  \centering
   \subfigure[Piecewise linear FEM with $h=2^{-5}$]{\includegraphics[height=1.7in,width=0.45\textwidth]{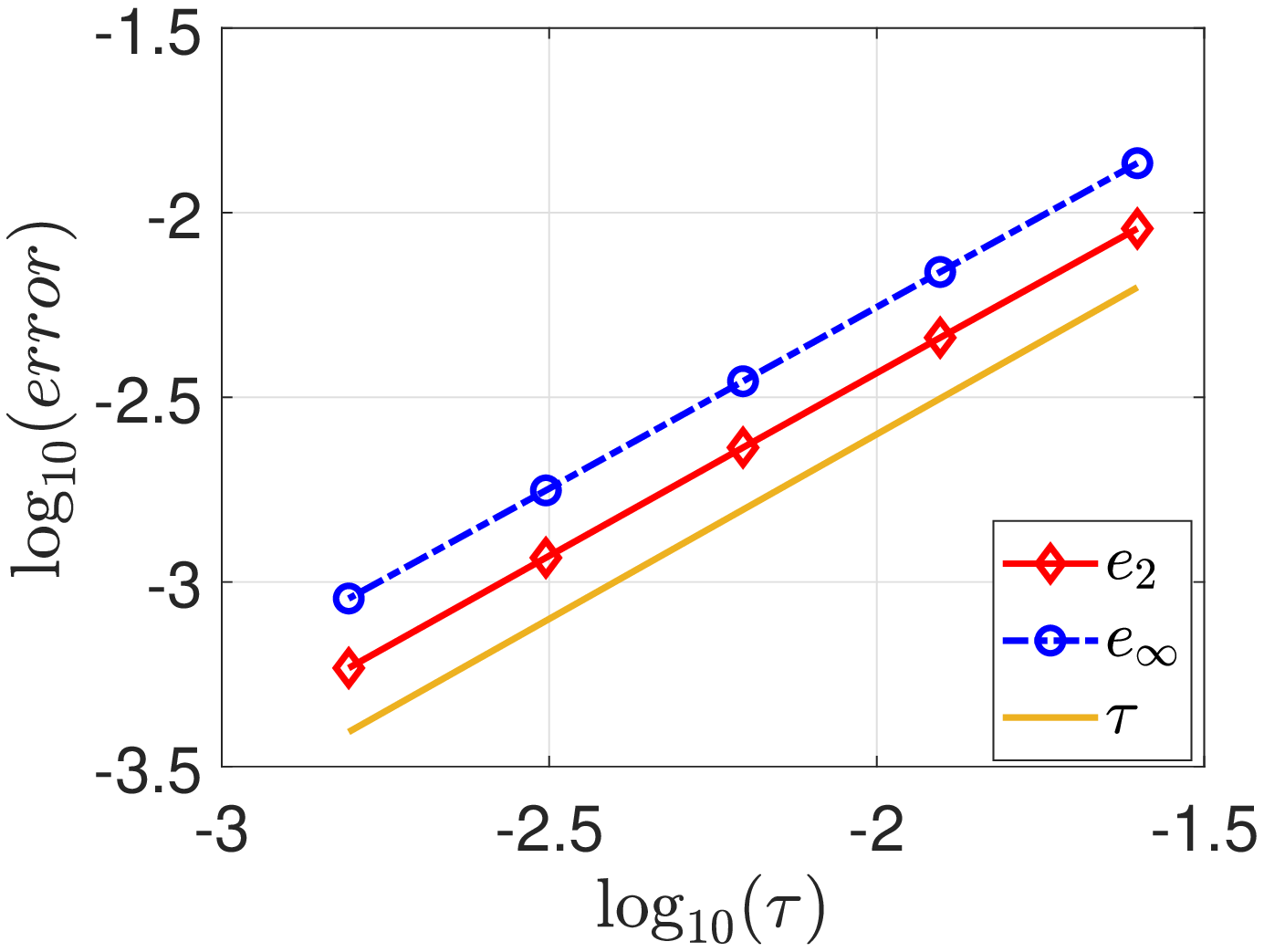}}
   \quad \subfigure[Piecewise quadratic FEM with $h=2^{-5}$]{\includegraphics[height=1.7in,width=0.45\textwidth]{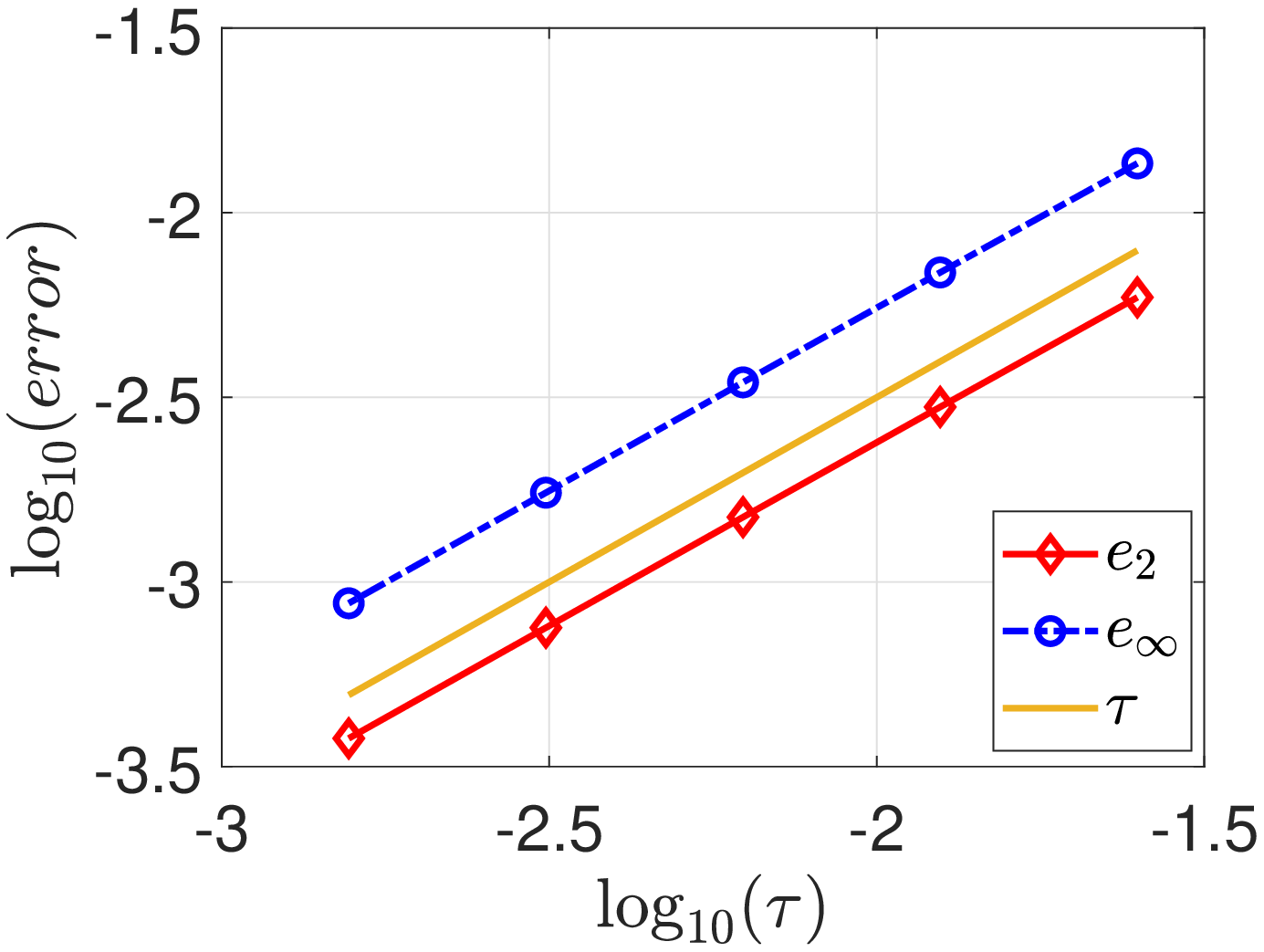}}
   \vspace{-0.05in}
    \caption{\small Convergence of the {first-order IMEX-FEM} scheme \eqref{lognlsfulldisweak} {in time} for the LogSE in 1D. }\label{logsetimeorderFEM1D}
\end{figure}
\begin{figure}[!h]
  \centering
   \subfigure[Piecewise linear FEM with $\tau=10^{-5}$]{\includegraphics[height=1.7in,width=0.45\textwidth]{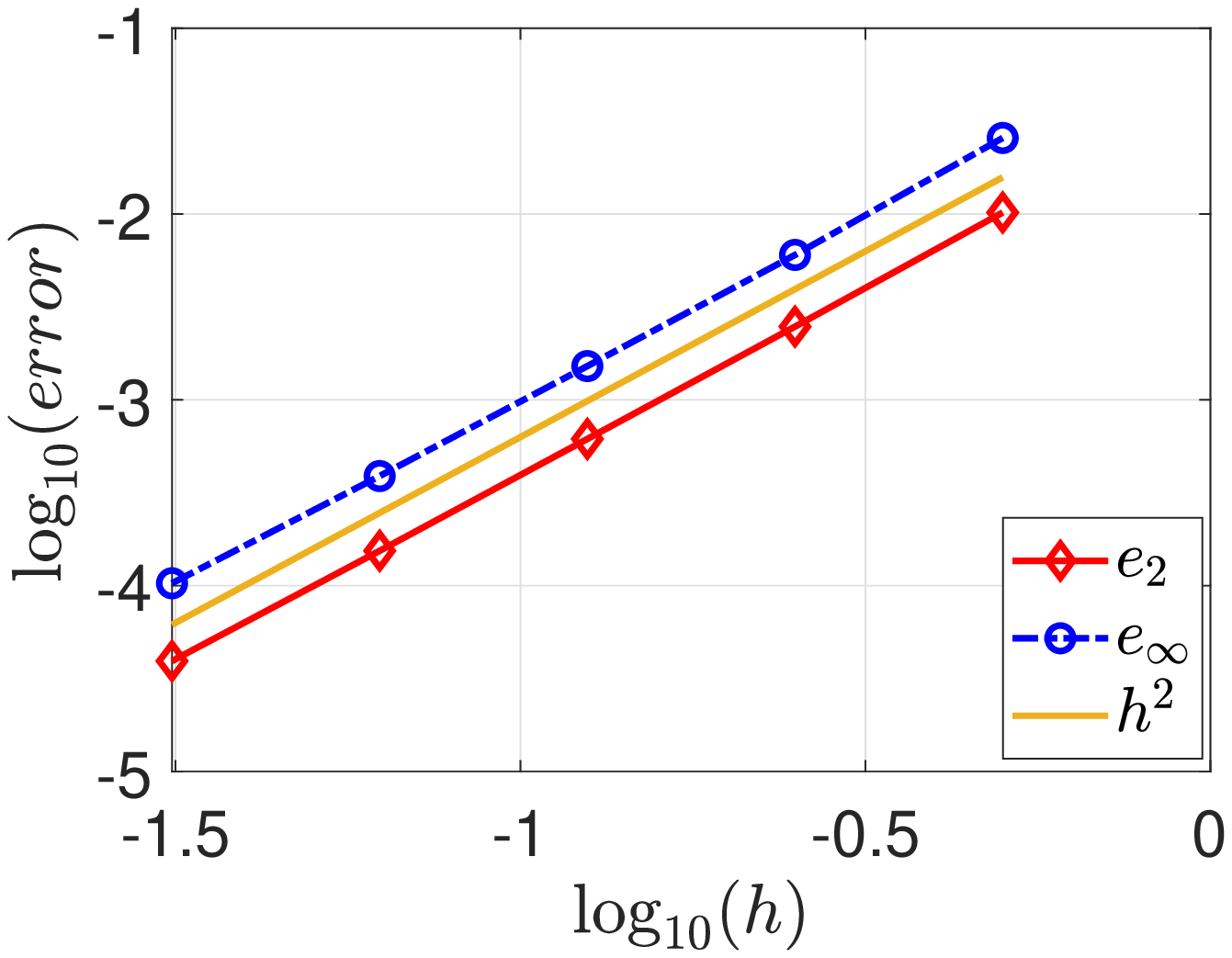}}
   \subfigure[Piecewise quadratic FEM with $\tau=10^{-5}$]{\includegraphics[height=1.7in,width=0.45\textwidth]{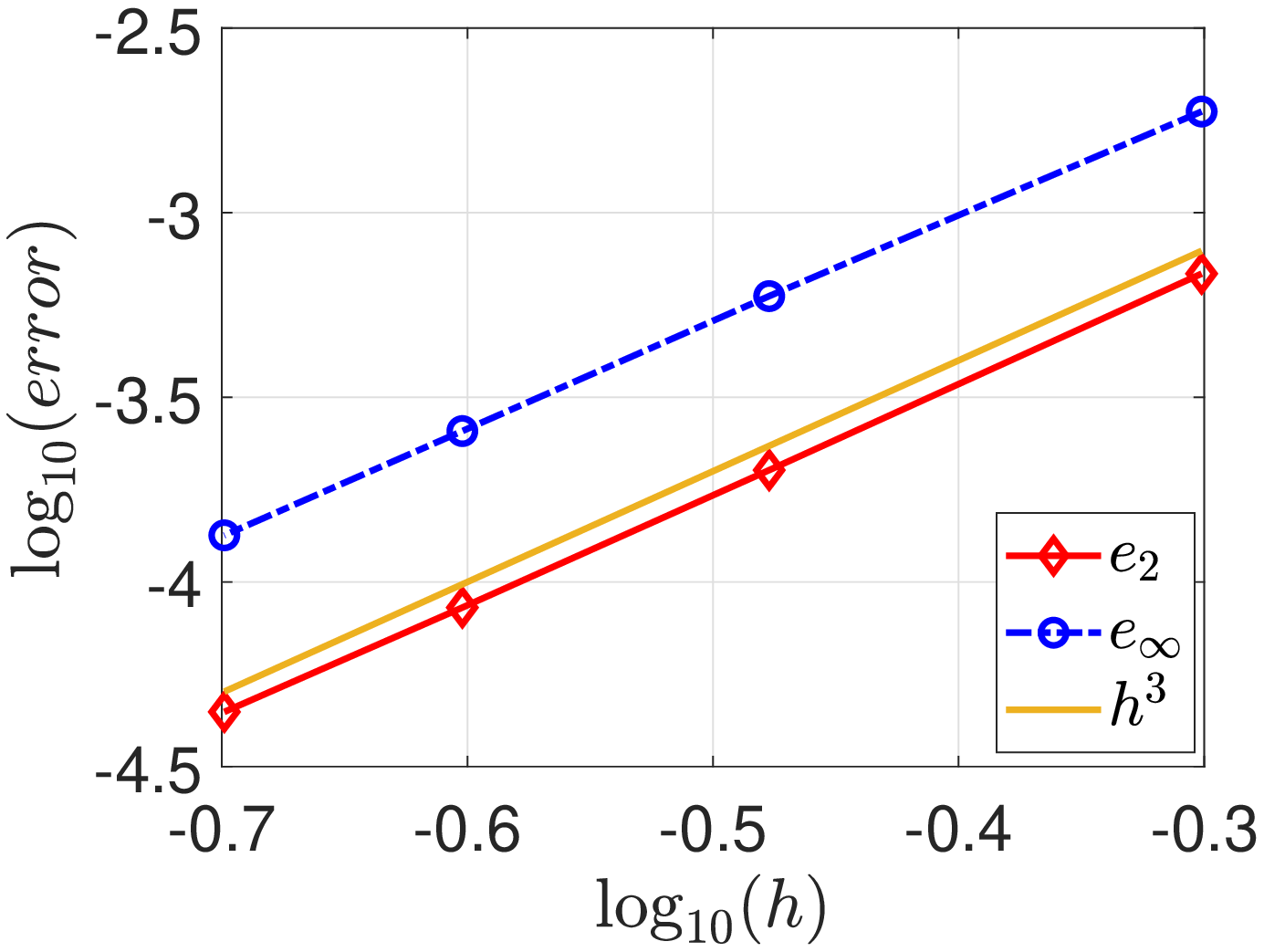}}
   \vspace{-0.05in}
    \caption{\small Convergence of the {first-order IMEX-FEM} scheme \eqref{lognlsfulldisweak} {in space} for the LogSE in 1D. }\label{logsespaceorderFEM1D}
\end{figure}


{In Figure \ref{logseorderFEM2D}, we depict  the errors and convergence orders of the proposed scheme for the LogSE in 2D. Here we take  $d=2$  in \eqref{2dGex} with the other parameters being the same as the one-dimensional tests.  To demonstrate the convergence rate in time, we fix $h=2^{-5}$ and take the time step as $\tau_j=0.01\times2^{-j},\,j=1,\ldots,5$ for the piecewise linear FEM (see Figure \ref{logseorderFEM2D} left). To examine the convergence order in space-time, we take $\tau_j=h^{2}_j$ with $h_j=\frac{1}{20+4j},\,j=1,\ldots,4$.
The numerical results in Figure \ref{logseorderFEM2D} demonstrate the expected convergence orders.}
\begin{figure}[!h]
  \centering
   \subfigure[{ Two-dimensional test with               $h=2^{-5}$}]{\includegraphics[height=1.7in,width=0.45\textwidth]{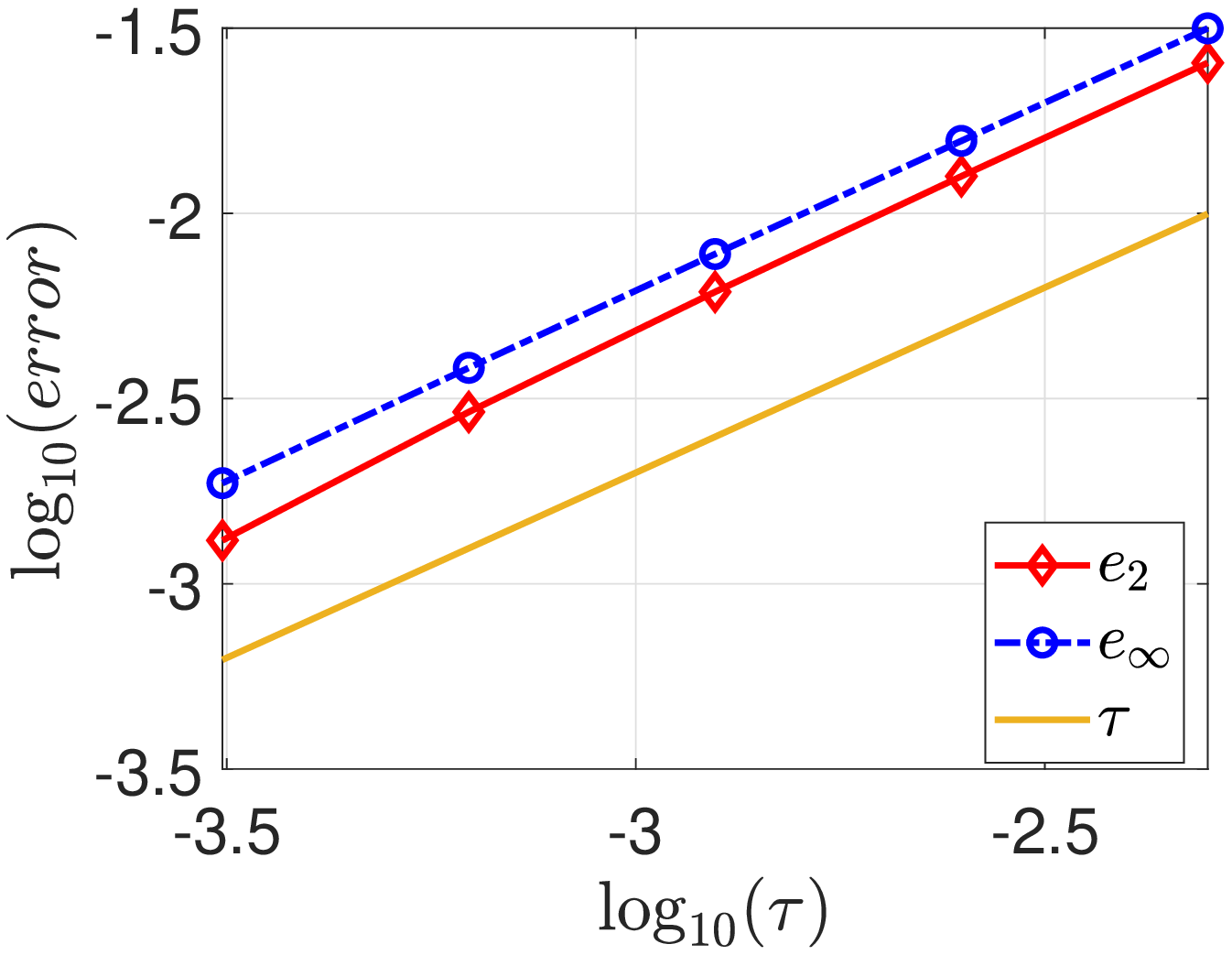}}
   \subfigure[{Two-dimensional test with $\tau=h^{2}$}]{\includegraphics[height=1.7in,width=0.45\textwidth]{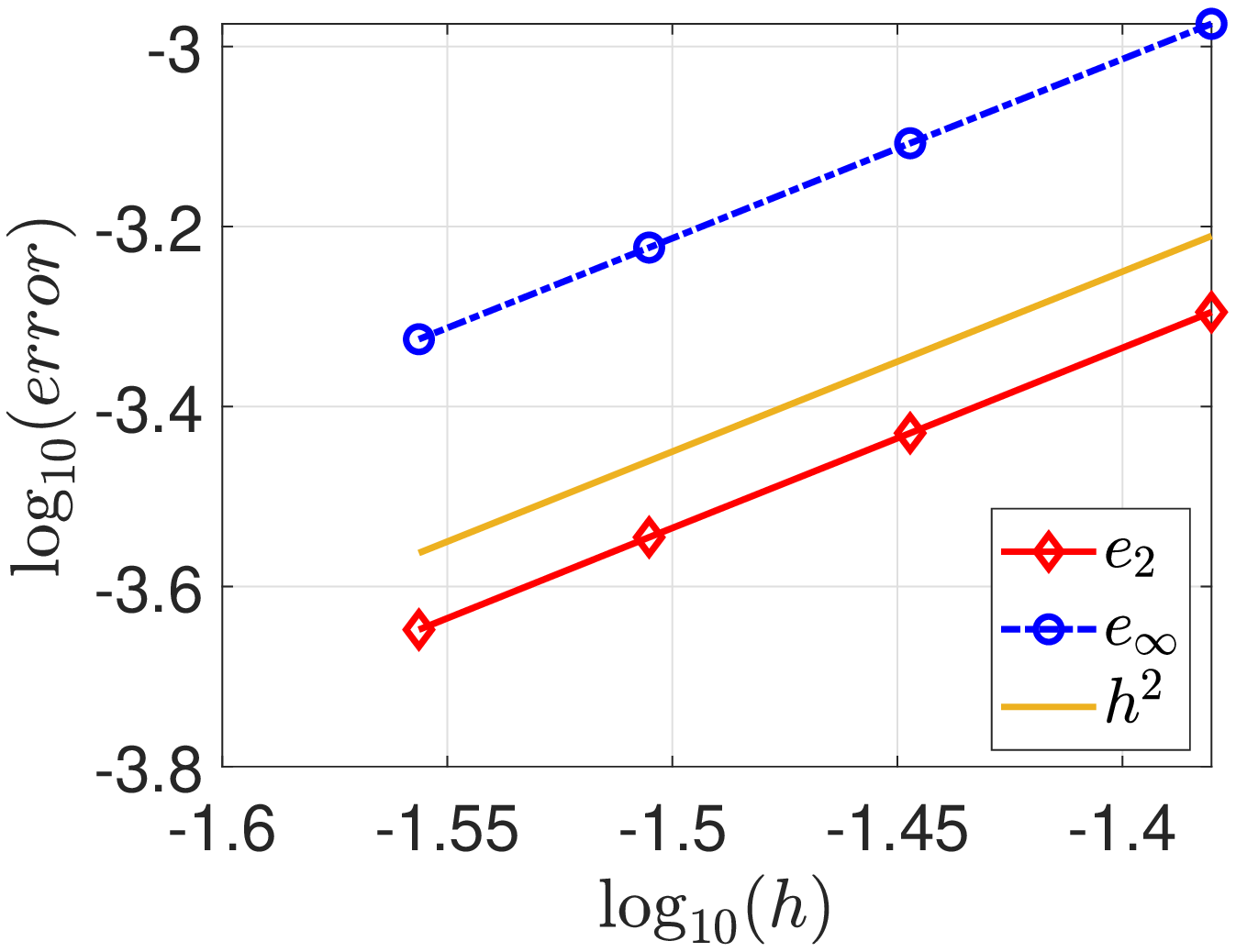}}
   \vspace{-0.05in}
    \caption{\small Convergence of the {first-order IMEX} piecewise linear FEM scheme \eqref{lognlsfulldisweak} {in time  with fixed $h$ (left) and space-time with $\tau=h^2$ (right)} for the LogSE in 2D. }\label{logseorderFEM2D}
\end{figure}

\begin{remark}\label{e2einfty}
{\em Observe from the above numerical results that $e_2$ and $e_\infty$ have a similar convergence order. In fact, for such a  smooth solution, its FEM approximation  in $L^2$-norm and $L^\infty$-norm differ from a log-factor. For example,   we find form \cite[Theorem 1.4]{Thomee2006Galerkin} that in the piecewise linear case, we have
$$\|R_{h}u-u \|_\infty+\|  \mathcal{I}_{h} u -u\|_{ \infty} \le C \ell_h h^2 \| u \|_{W^{2,\infty}(\Omega)}.$$
Although we could not rigorously derive the $L^\infty$-estimate for the scheme, we observe similar convergence behaviours. \qed
 }
\end{remark}


In the second test, we study the dynamics of the LogSE with $d=1$ by choosing the initial data as a combination of two Gaussian functions
as in \cite[Example 3]{Bao2019Regularized}:
\begin{equation}\label{Gaini}
u_{0}(x)=\sum_{k=1}^{2} {\rm exp}\Big({-\frac{a_{k}}{2}\left(x-x_{k}\right)^{2}+\ri \zeta_{k} x}\Big), \quad x \in \mathbb{R}.
\end{equation}
  with the velocity $\zeta_k$ and the initial location $x_k$. Here, we can consider the following cases:
 \begin{enumerate}\label{num:cases}
   \item[(i)] {$a_1=a_2=1, \zeta_1=\zeta_2=0, x_1=-5, x_2=5;$}
   \item[(ii)] {$a_1=a_2=1, \zeta_1=\zeta_2=0, x_1=-2, x_2=2;$}
   \item[(iii)] {$a_1=a_2=1, \zeta_1=2, \zeta_2=-2, x_1=-30, x_2=30.$}
 \end{enumerate}
We solve this problem on $\Omega=(-40,40)$ with the time step size ${10^{-4}}$, mesh size $h=0.05$  and $\lambda=-1$ and with zero boundary conditions. In Figure \ref{logseu1DFEM}, we depict the time evolution of the numerical solutions and  ``mass''. {For static Gaussons (i.e., $\zeta_k = 0$ in \eqref{Gaini}), if they were initially well-separated, the
two Gaussons will stay stable as separated static Gaussons (see Figure \ref{logseu1DFEM} (top row)) with
density profile unchanged. However, when they get closer, the two Gaussons contact and undergo
attractive interactions. They move to each other, collide and stick together later. Shortly,
the two Gaussons separate and swing like pendulum (see Figure \ref{logseu1DFEM}  (second row)). For moving Gaussons, the two Gaussons basically move at constant velocities and preserve their profiles in density exactly, during the interaction, there occurs oscillation, and after collision, the velocities of Gaussons change. The two Gaussons oscillate like breathers and separate completely at last (see Figure \ref{logseu1DFEM} (last row)).}
\begin{figure}[!ht]
  \centering
  \subfigure{\includegraphics[width=0.30\textwidth]{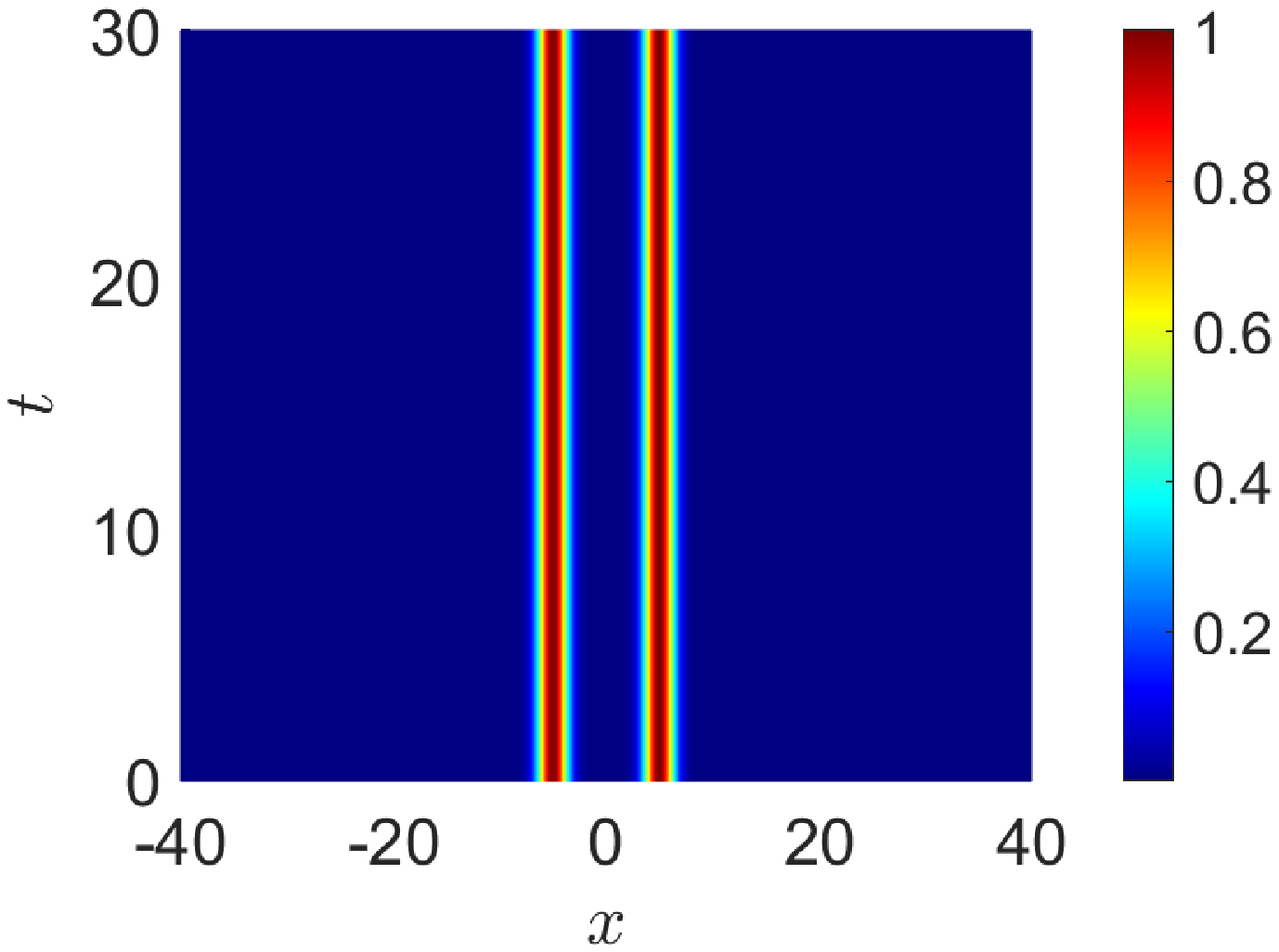}}
   \subfigure{\includegraphics[width=0.30\textwidth]{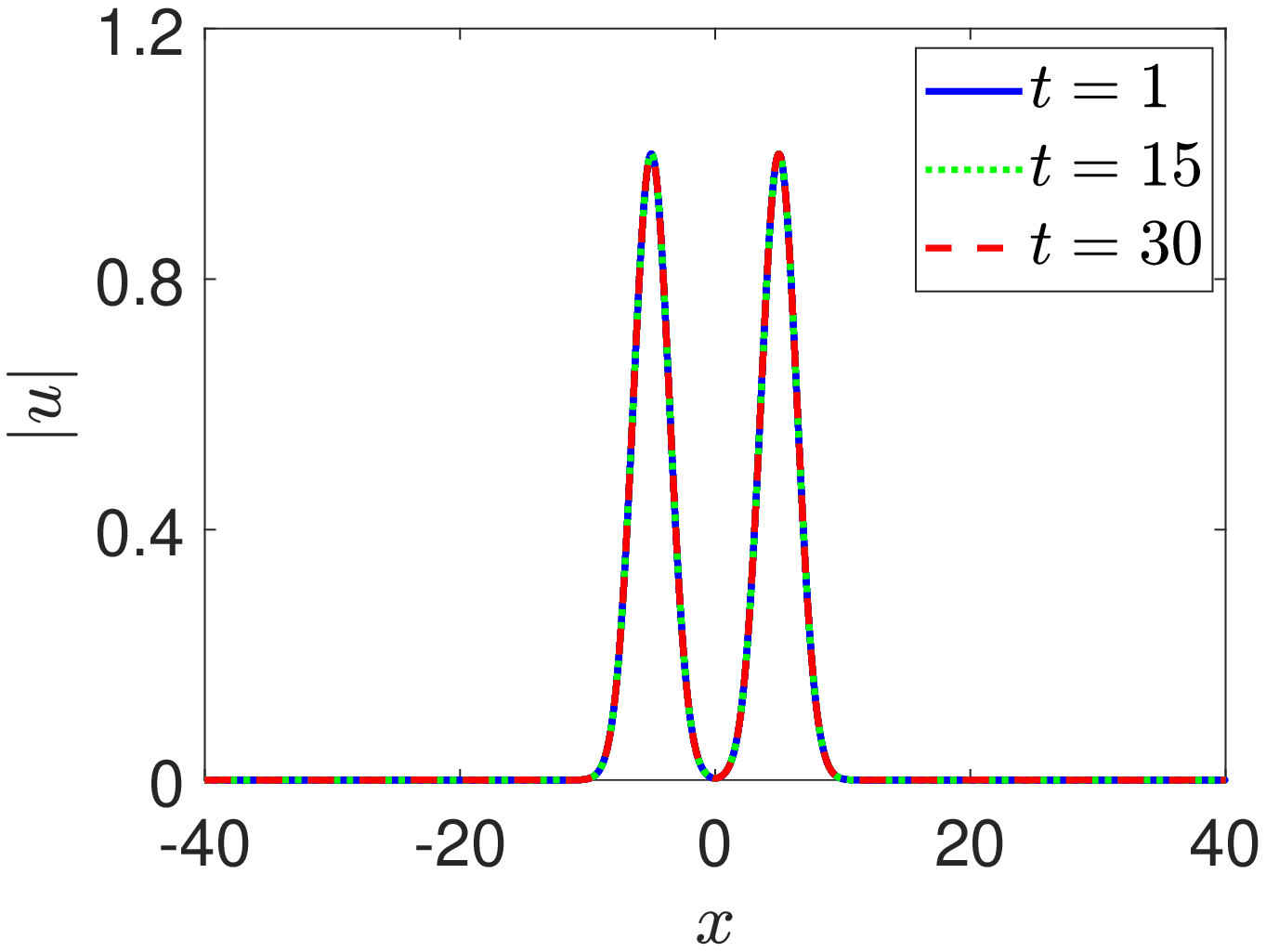}}
   \subfigure{\includegraphics[width=0.30\textwidth]{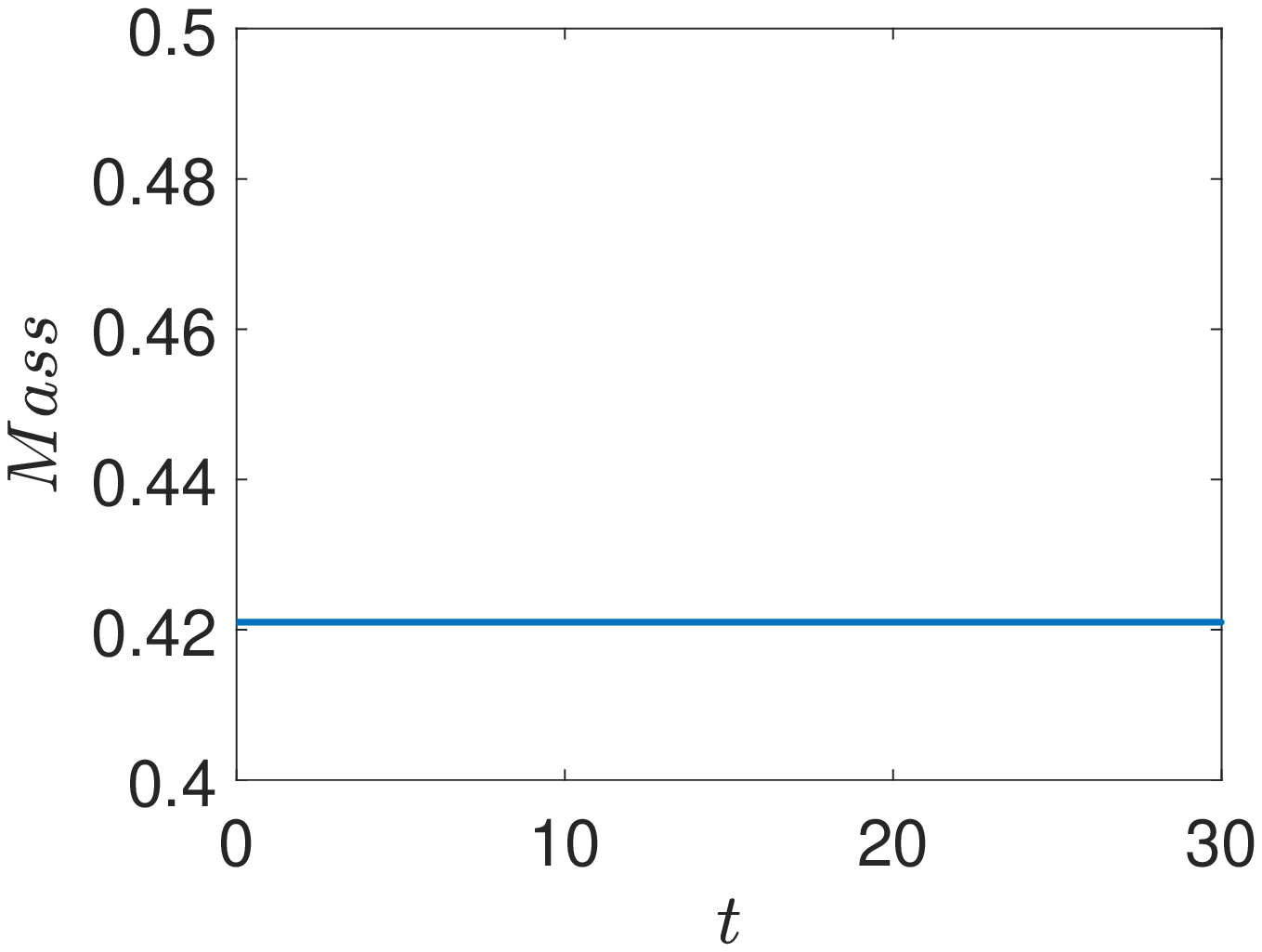}}
   \subfigure{\includegraphics[width=0.30\textwidth]{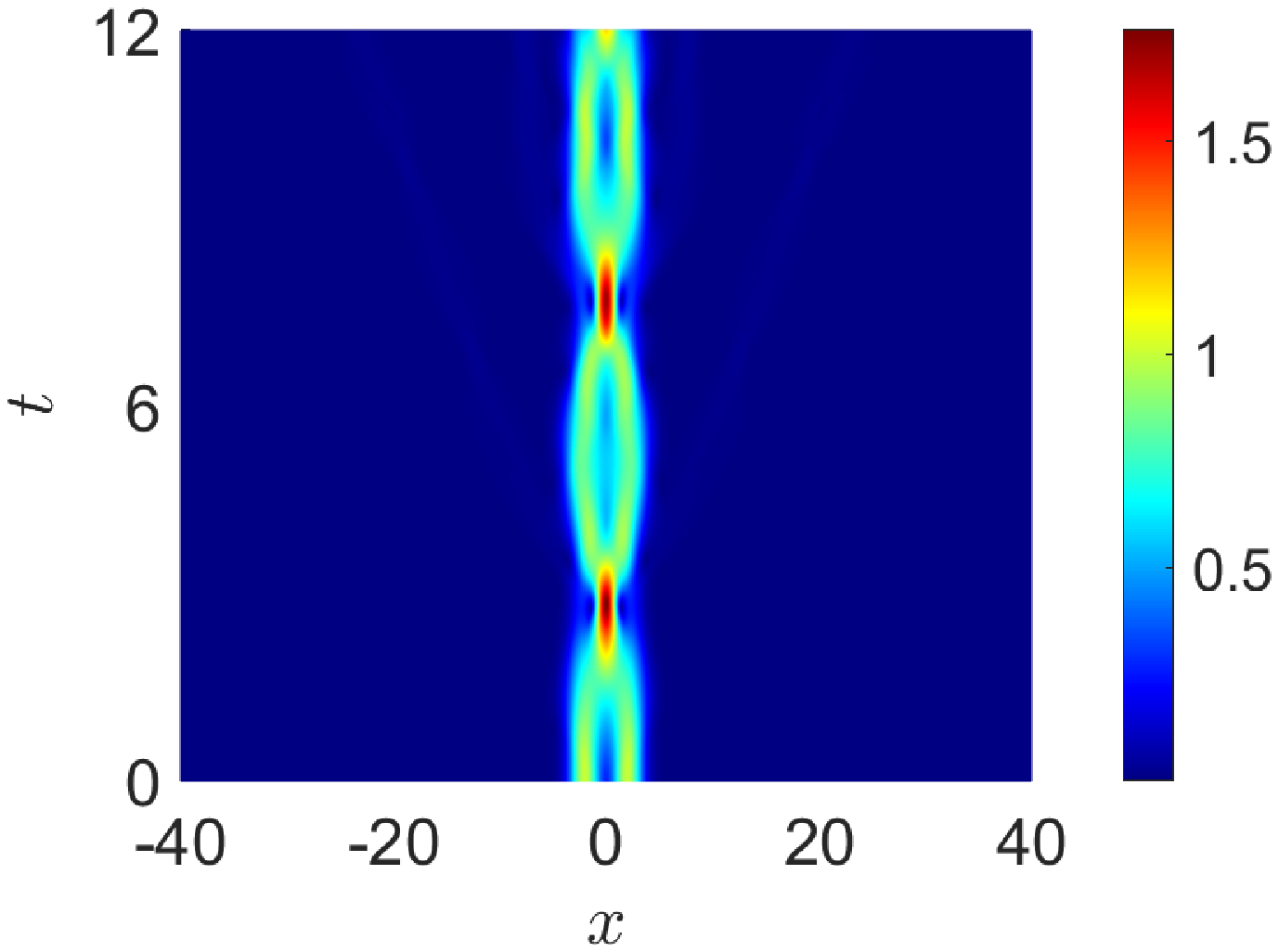}}
   \subfigure{\includegraphics[width=0.30\textwidth]{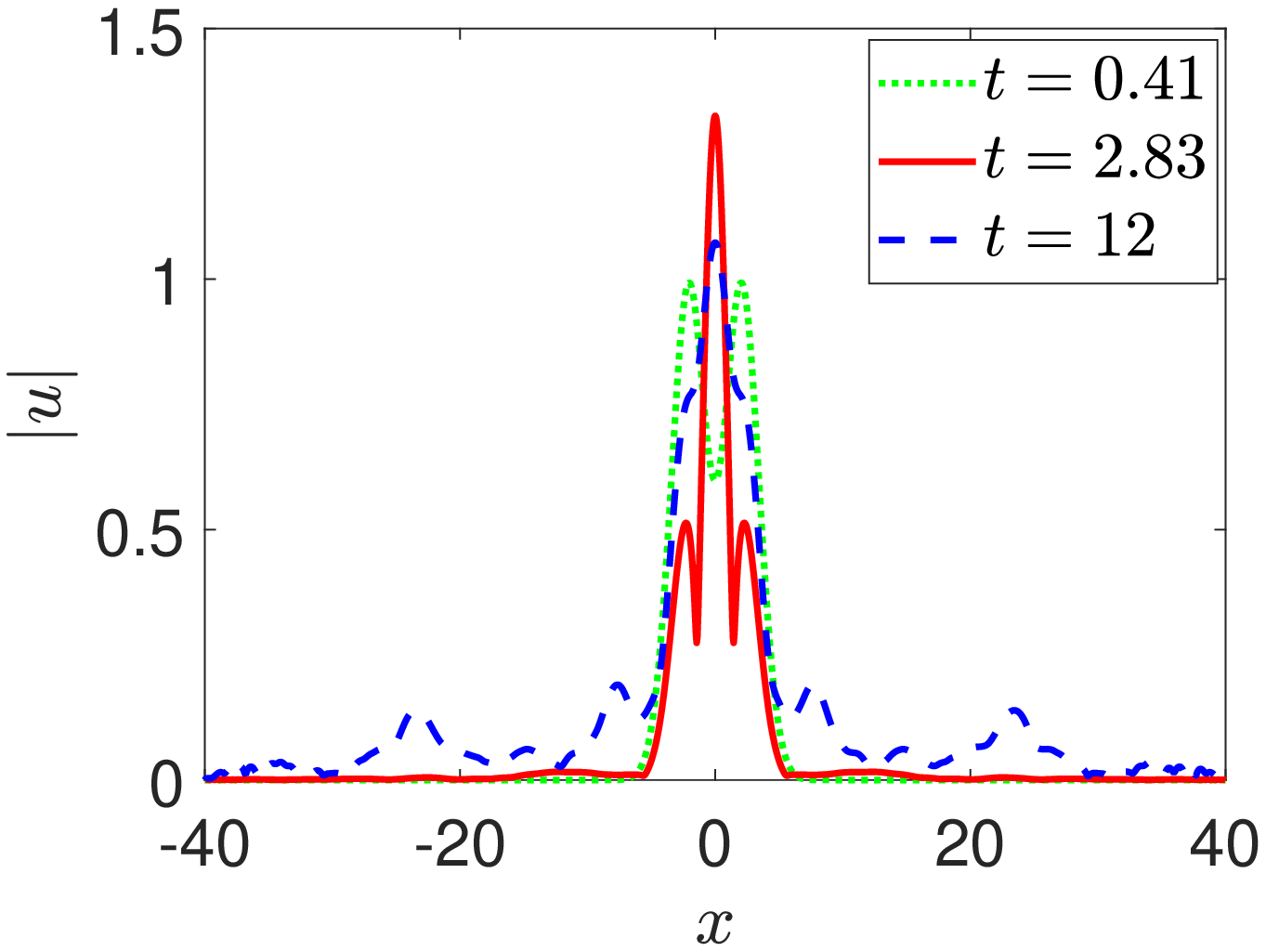}}
   \subfigure{\includegraphics[width=0.30\textwidth]{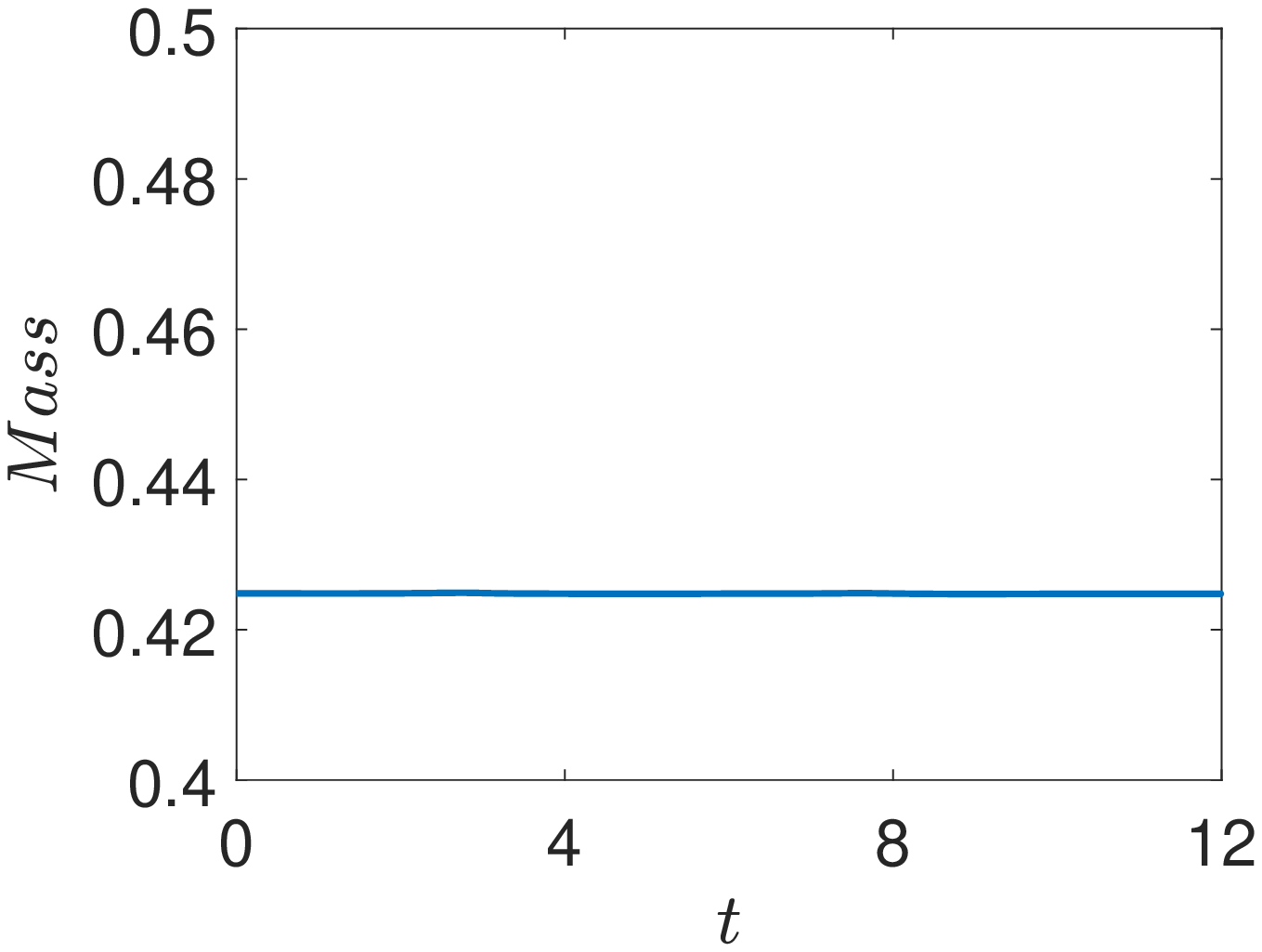}}
   \subfigure{\includegraphics[width=0.30\textwidth]{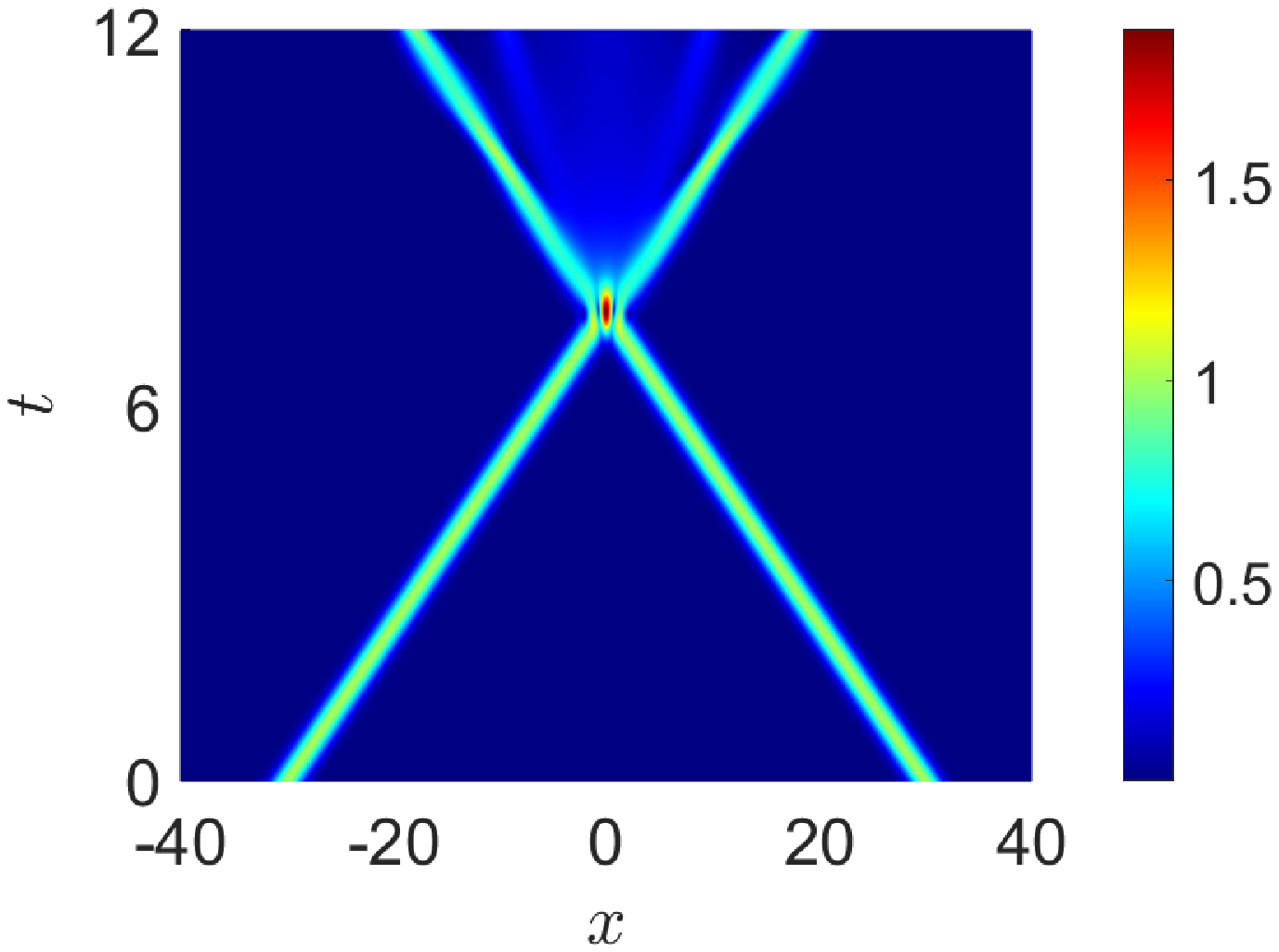}}
   \subfigure{\includegraphics[width=0.30\textwidth]{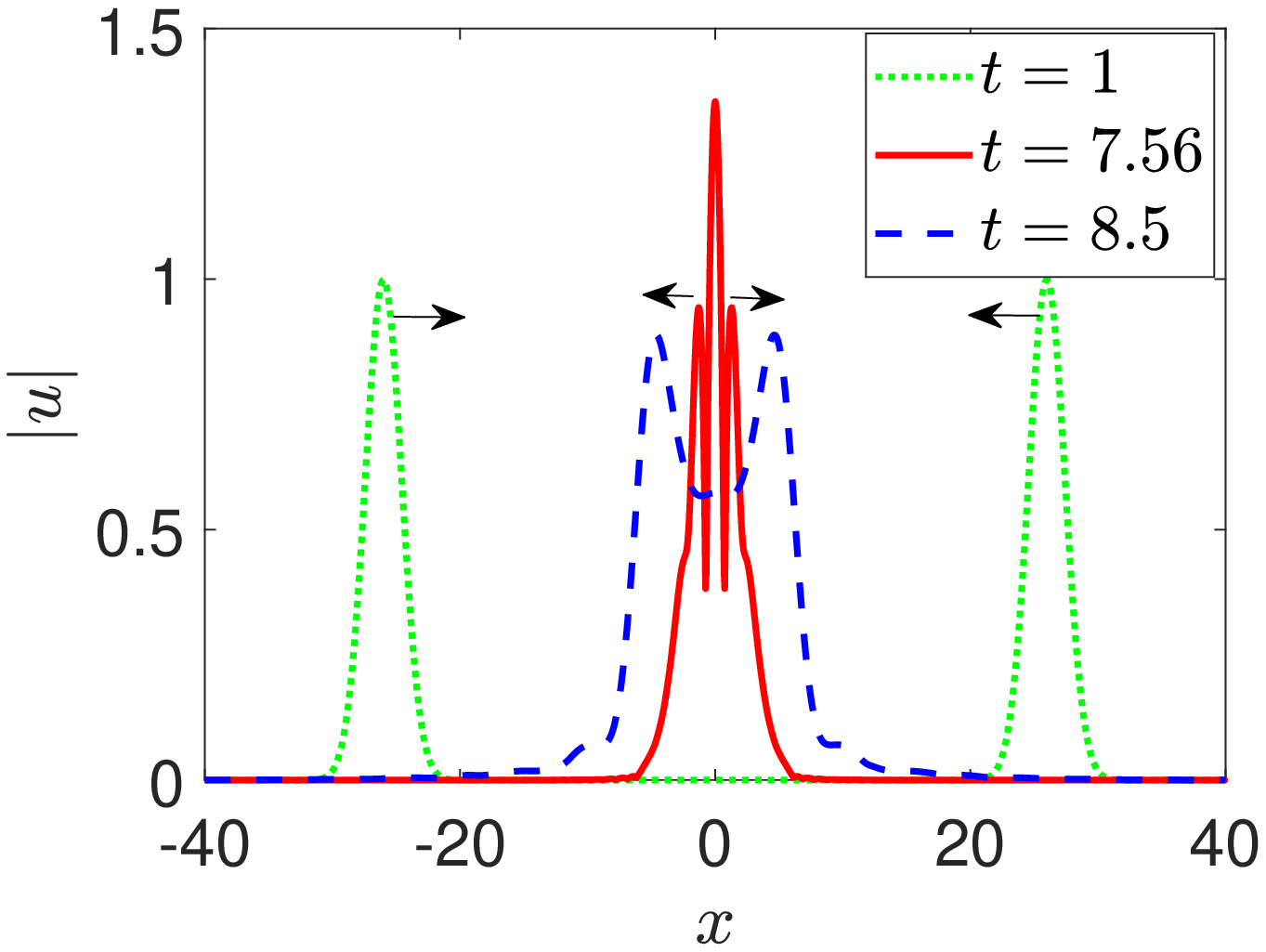}}
   \subfigure{\includegraphics[width=0.30\textwidth]{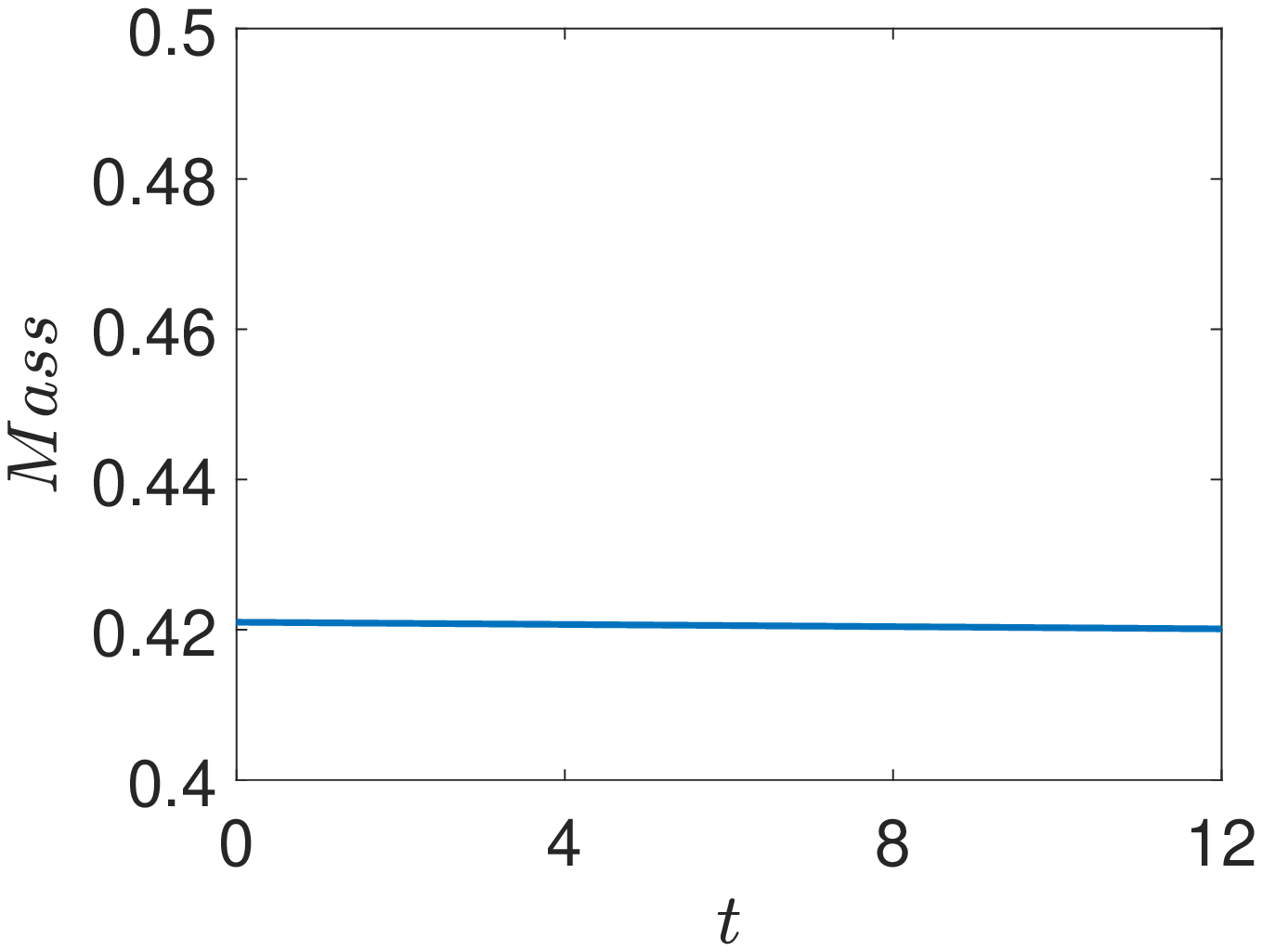}}
   \caption{\small Plots of $|u|$ (first column), $|u|$ at different times (second column), and mass (third column) for different parameters, i.e.,  Cases (i) - (iii) (from top to bottom) given above.
    }\label{logseu1DFEM}
\end{figure}

We next take the initial data as in \cite[Example 3]{Paraschis2023FixedpointLogSE}:
\begin{equation}
u_0(x,y)=\tanh(x)\tanh(y)\exp(-x^2-y^2),
\end{equation}
and  solve the problem on $\Omega= (-10,10)^2$ with time step ${\tau=10^{-4}}$ and mesh size $h=0.1$ with zero boundary conditions.
In Figure  \ref{logseu2DFEM}, we plot the numerical solutions at  $t=0, \, 0.25, \,0.5,$ where the four peaks expand as time evolves, and the magnitude decreases and  the supports span to preserve the mass.
\begin{figure}[!ht]
  \centering
   \subfigure[$T=0$]{\includegraphics[width=0.3\textwidth]{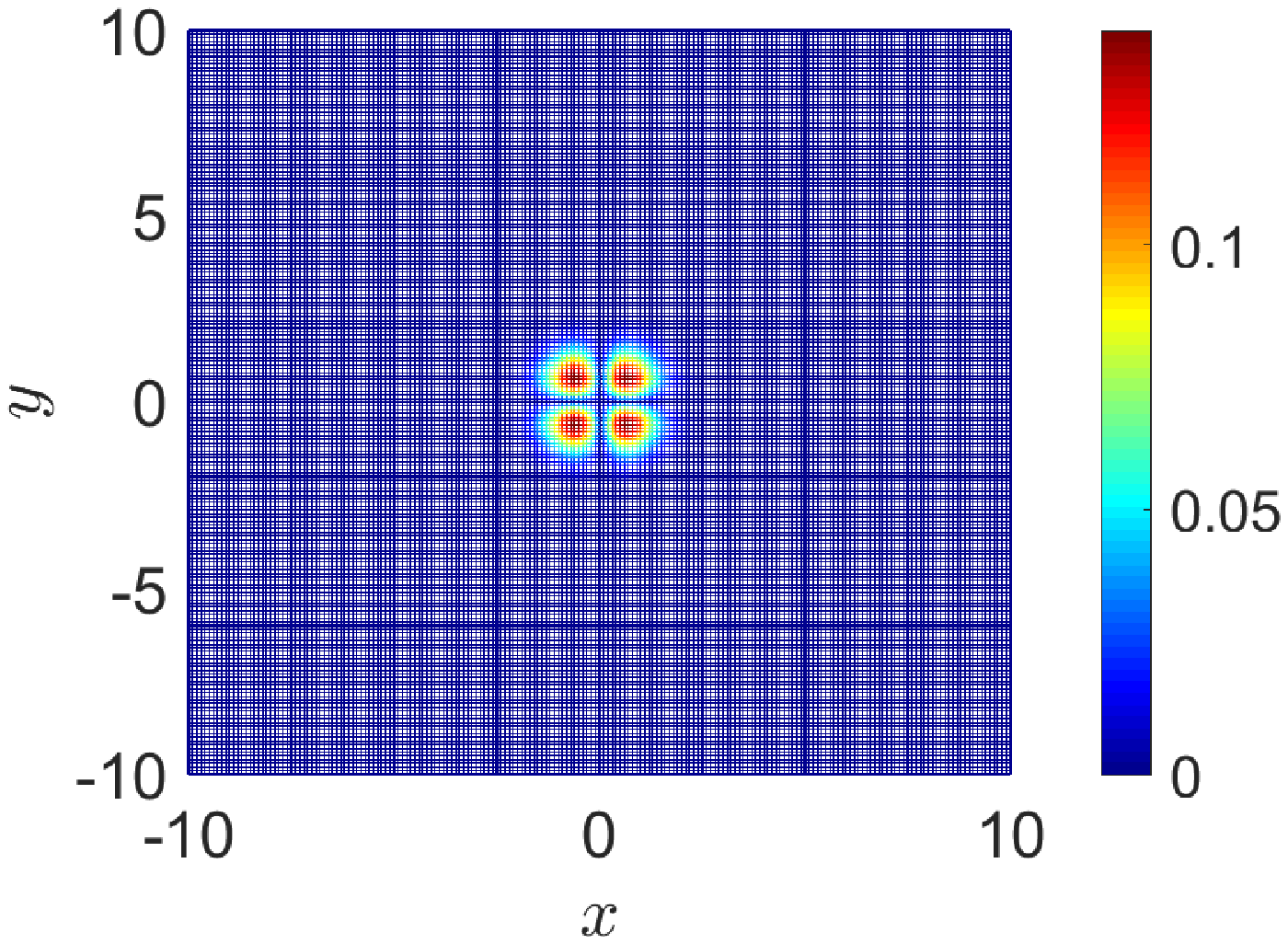}}
   \subfigure[$T=0.25$]{\includegraphics[width=0.3\textwidth]{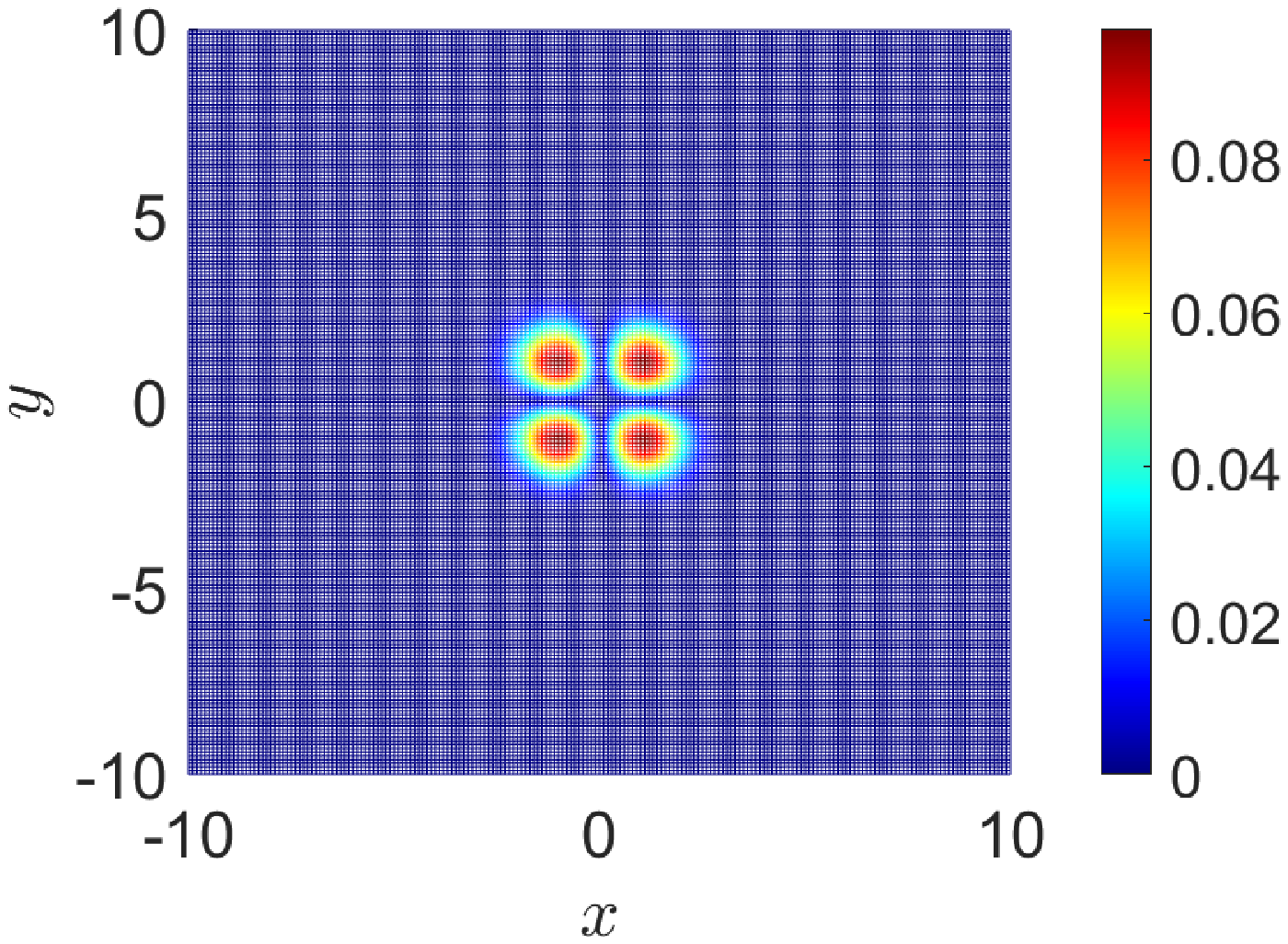}}
   \subfigure[$T=0.5$]{\includegraphics[width=0.3\textwidth]{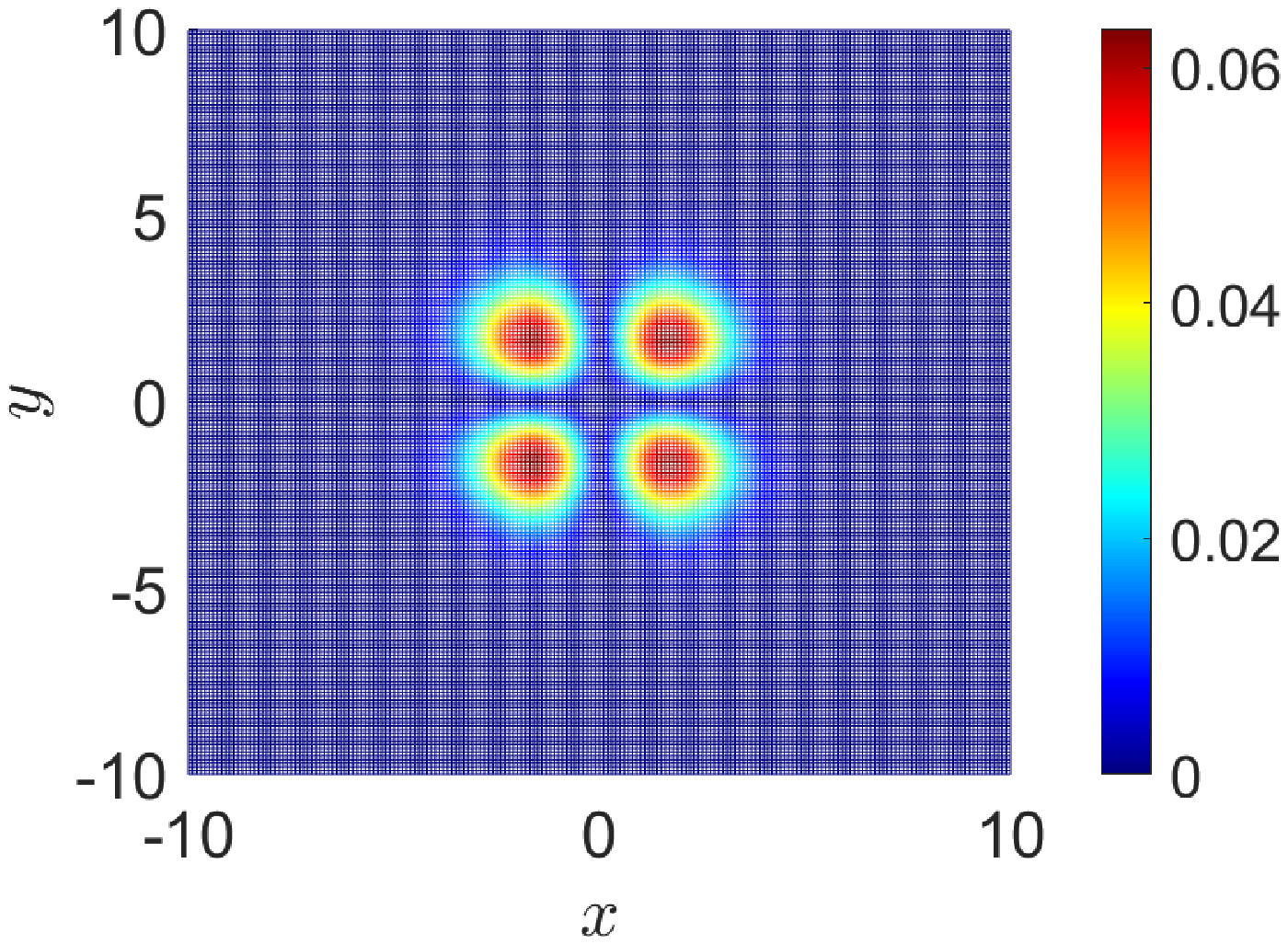}}
   \caption{The dynamics of the numerical solution of the LogSE at different times.}\label{logseu2DFEM}
\end{figure}
\section{Concluding remarks}
In this paper, we analyzed for the first time a non-regularized IMEX scheme for the LogSE  based on  the H\"older continuity of the nonlinear term and a  nonlinear Gr\"{o}nwall's inequality.  Compared with its counterpart with a regular nonlinear term, there appeared an extra $|\!\ln h|$-factor and a  mild reduction of convergence order,  which seemed inevitable.

The logarithmic PDEs have attracted much recent interest in analysis and computation. There are many issues are worthy of deep and future investigations, e.g., the mass and energy-preserving schemes for the LogSE and its variants including the LogSE with regular or singular potentials. This is indeed our first endeavour in exploring such  singular PDEs. 

\bigskip

\noindent{\bf Acknowledgment:}\;  The authors would like to thank the anonymous referees for the valuable comments  that have   led to  significant improvements of this work.  In particular, we are grateful to the referee's suggestion on the analysis of the linearised scheme
$$
{\rm i} \frac{u^{n+1}-u^n}{\tau} +\Delta u^{n+1} =\lambda u^n \ln (\tau+|u^n|^2).
$$
This motivated us to consider the scheme \eqref{lognlsfulldisweak} with a direct discretisation of the nonlinear term as $\lambda u^n \ln |u^n|^2$.

\begin{appendix}
	\setcounter{equation}{0}
	\renewcommand{\theequation}{A.\arabic{equation}}
		\section{Proof of Proposition \ref{ExistUnique}}\label{AppendixA}
We resort to the abstract framework (see e.g., \cite[Lemma 3.1]{akrivis_fully_1991}):
{\em Let $\{{\mathbb H},(\cdot,\cdot)_{\mathbb H}\}$ be a finite-dimensional inner product space with the associated norm $\| \cdot \|_{\mathbb H},$ and let the functional $g: {\mathbb H} \mapsto {\mathbb H}$ be  continuous.  Assume that there exists $\sigma>0$ such that for every $z \in {\mathbb H}$ with $\| z \|_{\mathbb H}= \sigma,$ there holds $\Re(g(z), z)_{\mathbb H}\ge 0$. Then there exists some $z_{*} \in {\mathbb H}$ such that $g( z_{*} )=0$ and $\| z_{*} \|_{\mathbb H} \le \sigma$.}

To adapt it to this setting,  we define a continuous linear functional $g : V_h^0 \mapsto V_h^0$ of the form
\begin{equation*}
g(v): = v  - \ri \tau \Delta v + 2\ri \lambda \tau  f(u_h^n) - u_h^n.
\end{equation*}
Let ${\mathbb H}=L^2(\Omega)$ and set $\| v \| = \sigma$. By using 
the Cauchy-Schwarz inequality, we have
\begin{equation*}
\begin{split}
\Re (g(v), v)&= \Re\big\{ \| v \|^2  + \ri \tau \| \nabla v \|^2 + 2\ri \lambda \tau ( f(u_h^n), v ) - (u_h^n, v) \big\} \\
& = \| v \|^2 -2 \lambda \tau \Im ( f(u_h^n), v ) - \Re (u_h^n, v) \\
& \ge  \big( \| v \| - 2|\lambda| \tau  \| f(u_h^n)\| -  \| u_h^n \| \big)\| v \|.
\end{split}
\end{equation*}
We next show that $ \| f(u_h^n) \|$ is bounded.  Following  \cite[sec. 2]{Bao2019Error} and using the fundamental inequality  
\begin{equation*}
\ln(x) \le \frac{ x^{\eta} }{  \eta  \rm{e} }, \quad  x \ge 1,\;\;\; 0< \eta \le 1,
\end{equation*}
and the interpolation inequality
$$
\|u\|_{L^p(\Omega)}\le c\big(\|u\|^{1-\beta} \|\nabla u\|^{\beta}+\|u\|\big),\quad \frac 1 p=\frac 1 2 -\frac{\beta} d,\;\; \beta\in [0,1),
$$
with $p=2(1+\eta),$ we derive  
\begin{equation*}
\begin{split}
\int_{\Omega} |u_h^{n}|^2 \big|\! \ln  | u_h^n | \big|^2 {\mathrm{d}} x &\le \frac{1}{ \eta {\mathrm{e} } }\int_{ \{x: |u_h^n|\ge 1 \} } |u_h^n|^{2+ 2\eta} { \mathrm{d} } x +  \frac{1}{ \eta {\mathrm{e} } } \int_{ \{x: |u_h^n| < 1 \} }  |u_h^n|^{2- 2\eta} { \mathrm{d} } x \\
& \le  \frac{1}{ \eta {\mathrm{e} } } \big( \| u_h^{n}\|_{ L^{2 + 2\eta}(\Omega) }^{2+ 2\eta} +|\Omega| \big) \\
& \le  C\big(\| u_h^{n} \|^{2+(2-d) \eta} \| \nabla u_h^n \|^{d\eta} + \| u_h^n\|^{2+2\eta}  + |\Omega| \big),
\end{split}
\end{equation*}
From the inverse inequality $\|\nabla u_h^n\|\le Ch^{-1}\|u_h^n\|$ (cf. \cite[Lemma 4.5.3]{Susanne2008Book}), we obtain
$$
\tau   \| f(u_h^n)\|\le C\tau(1+ \| u_h^n\|^{1+\eta-d\eta/2}   \| \nabla u_h^n \|^{d\eta/2})\le C\tau\big(1+ h^{-d\eta/2}\| u_h^n\|^{1+\eta}\big).
$$
We choose some
$$\sigma\ge \sigma_{\tau,h} = C\tau\big(1+ h^{-d\eta/2}\| u_h^n\|^{1+\eta}\big)+\|u_h^n\|, $$
so that $\Re (g(v), v)\ge 0.$ From the statement  stated in the beginning of the proof, we deduce that there exists $v_* \in V_h^0$ such that $\| v_* \| \le \sigma$ and $g(v_*) =0$.

We now show the uniqueness of the solution to \eqref{lognlsfulldisweak}.
It is evident that the solution $v_*$ satisfies
\begin{equation*}\label{unique02}
\frac{\ri}{ \tau }  ( v_* - u_h^n, v_h) -  ( \nabla v_*, \nabla v_h) = 2\lambda (f(u_h^n), v_h ),\quad \forall v_h\in V_h^0.
\end{equation*}
Let $u_h^{n+1}$ be another solution of \eqref{lognlsfulldisweak}. Then we have the error equation
%
\begin{equation*}\label{unique03}
\frac{\ri}{ \tau } (u_h^{n+1} - v_*, v_h)  -( \nabla ( u_h^{n+1} - v_*), \nabla v_h) = 0.
\end{equation*}
Taking $v_h =  u_h^{n+1} - v_*,$ we find from the resulting equation that
\begin{equation*}
\| u_h^{n+1} - v_*\| =\|\nabla ( u_h^{n+1} - v_*)\|= 0,
\end{equation*}
which implies $u_h^{n+1} = v_*$.

\setcounter{equation}{0}
	\renewcommand{\theequation}{B.\arabic{equation}}
	\section{Proof of Lemma \ref{Tnorder}}\label{AppendixC}

	Using the Taylor expansion and applying some simple substitutions to the integral remainder, we have
	\begin{equation*}
		\mathcal{T}^{n}= {\rm{i}}  \tau\int^{1}_{0}(1-s)u_{tt}(\cdot,t_n+s\tau)\,\mathrm{d} s +  \tau  \int^{1}_{0} \Delta u_{t}(\cdot,t_n+s\tau)\, \mathrm{d}s,
	\end{equation*}
	where  we understand $u_t, u_{tt}$ are the partial derivatives with respect to the second variable of $u.$
	From the Cauchy-Schwarz inequality, we derive
	\begin{equation*}
		\begin{split}
			\|\mathcal{T}^{n}\|^2
			&\le {2 }\tau^2\int_{\Omega}\Big|\int^{1}_{0}(1-s)u_{tt}(x,t_n+s\tau)\mathrm{d}s\Big|^2\mathrm{d}x+ {2} \tau^2 \int_{\Omega}\Big|\int^{1}_{0}\Delta u_{t}(x,t_n+s\tau)\mathrm{d}s\Big|^2\mathrm{d}x\\
			&\le{2} \tau^2\int^{1}_{0}(1-s)^2\mathrm{d}s\int^{1}_{0}\int_{\Omega}\big|u_{tt}(x,t_n+s\tau)\big|^2\mathrm{d}x\mathrm{d}s+ {2}  \tau^2 \int^{1}_{0}\int_{\Omega}\big|\Delta u_{t}(x,t_n+s\tau)\big|^2\mathrm{d}x\mathrm{d}s\\
			&\le  { \frac{2}{3} } \tau^2 \int^{1}_{0}\|u_{tt}(\cdot,t_n+s\tau)\|^2\mathrm{d}s+ {2} \tau^2 \int^{1}_{0}\|\Delta u_{t}(\cdot,t_n+s\tau)\|^2\mathrm{d}s\\
			&\le {\frac{2}{3} } \tau^2 \max_{ t\in  [t_n, t_{n+1}] }\|u_{tt}\|^2+  {2 } \tau^2 \max_{ t\in [t_n, t_{n+1}]   }\|\Delta u_{t}\|^2.
		\end{split}
	\end{equation*}
	This implies \eqref{tnest}.
\end{appendix}

\bibliographystyle{siam}

\bibliography{Full}
\end{document}